\documentclass[12pt]{amsart}  	

\usepackage[margin=1in]{geometry}  

\usepackage[all]{xy}						
\usepackage{amssymb}
\usepackage{amscd,latexsym,amsthm,amsfonts,amssymb,amsmath,amsxtra}
\usepackage[mathscr]{eucal}
\usepackage[colorlinks]{hyperref}
\usepackage{mathtools}

\usepackage{chngcntr}
\numberwithin{equation}{section}
\newcounter{keepeqno}

\hypersetup{linkcolor=black, citecolor=black}

\newcommand{\BA}{{\mathbb {A}}}

\newcommand{\BC}{{\mathbb {C}}}

\newcommand{\BL}{{\mathbb {L}}}

\newcommand{\BR}{{\mathbb {R}}}

\newcommand{\BZ}{{\mathbb {Z}}}

\newcommand{\CA}{{\mathcal {A}}}

\newcommand{\CC}{{\mathcal {C}}}
\newcommand{\CE}{{\mathcal {E}}}
\newcommand{\CF}{{\mathcal {F}}}

\newcommand{\CL}{{\mathcal {L}}}
\newcommand{\CM}{{\mathcal {M}}}

\newcommand{\CS}{{\mathcal {S}}}

\newcommand{\CW}{{\mathcal {W}}}

\newcommand{\CZ}{{\mathcal {Z}}}

\newcommand{\FN}{{\mathfrak {N}}}

\newcommand{\FR}{{\mathfrak {R}}}

\newcommand{\FX}{{\mathfrak {X}}}

\newcommand{\FZ}{{\mathfrak {Z}}}

\newcommand{\Fc}{{\mathfrak {c}}}

\newcommand{\Fn}{{\mathfrak {n}}}
\newcommand{\Fo}{{\mathfrak {o}}}
\newcommand{\Fp}{{\mathfrak {p}}}

\newcommand{\RG}{{\mathrm {G}}}

\newcommand{\RI}{{\mathrm {I}}}

\newcommand{\RM}{{\mathrm {M}}}

\newcommand{\RO}{{\mathrm {O}}}

\newcommand{\RS}{{\mathrm {S}}}

\newcommand{\RU}{{\mathrm {U}}}

\newcommand{\cusp}{{\mathrm{cusp}}}

\newcommand{\GJ}{{\mathrm{GJ}}}
\newcommand{\GL}{{\mathrm{GL}}}

\newcommand{\Id}{{\mathrm{Id}}}

\renewcommand{\Re}{{\mathrm{Re}}}
\newcommand{\reg}{{\mathrm{reg}}}
\newcommand{\Res}{{\mathrm{Res}}}

\newcommand{\SL}{{\mathrm{SL}}}

\newcommand{\Span}{{\mathrm{Span}}}

\newcommand{\tr}{{\mathrm{tr}}}

\newcommand{\ud}{\,\mathrm{d}}
\newcommand{\vol}{{\mathrm{vol}}}

\newcommand{\wt}{\widetilde}

\newcommand{\bs}{\backslash}

\def\alp{{\alpha}}
\def\ac{\mathrm{ac}}

\def\bks{{\backslash}}
\def\del{{\delta}}
\def\Del{{\Delta}}
\def\diag{{\rm diag}}

\def\veps{{\varepsilon}}

\def\sig{{\sigma}}

\def\std{\rm std}
\def\ome{{\omega}}
\def\Ome{{\Omega}}

\def\gam{{\gamma}}

\def\wb{\overline} 

\def\vpi{\varpi}

\def\vphi{\varphi}

\def\pv{\mathrm{pv}}

\newtheorem{thm}{Theorem}[section]

\newtheorem{prp}[thm]{Proposition}
\newtheorem{lem}[thm]{Lemma}
\newtheorem{cor}[thm]{Corollary}
\newtheorem{ass}[thm]{Assumption}
\newtheorem{cnj}[thm]{Conjecture}

\makeatletter

\newcommand{\Rmnum}[1]{\expandafter\@slowromancap\romannumeral #1@}
\makeatother

\DeclarePairedDelimiter{\floor}{\lfloor}{\rfloor}

\begin{document}

\title[Fourier Operators on $\GL_1$ and Langlands $\gamma$-functions]
{Certain Fourier Operators on $\GL_1$ and Local Langlands Gamma functions}

\author{Dihua Jiang}
\address{School of Mathematics\\
University of Minnesota\\
Minneapolis, MN 55455, USA}
\email{dhjiang@math.umn.edu}

\author{Zhilin Luo}
\address{Department of Mathematics\\
University of Chicago\\
Chicago IL, 60637, USA}
\email{zhilinchicago@uchicago.edu}



\thanks{The research of this paper is supported in part by the NSF Grant DMS--1901802.}

\date{\today}

\begin{abstract}
For a split reductive group $G$ over a number field $k$, let $\rho$ be an $n$-dimensional complex
representation of its complex dual group $G^\vee(\BC)$. For any irreducible cuspidal automorphic
representation $\sig$ of $G(\BA)$, where $\BA$ is the ring of adeles of $k$, in \cite{JL21}, the authors
introduce the $(\sig,\rho)$-Schwartz space $\CS_{\sig,\rho}(\BA^\times)$ and $(\sig,\rho)$-Fourier
operator $\CF_{\sig,\rho}$, and study the $(\sigma,\rho,\psi)$-Poisson summation formula on $\GL_1$,
under the assumption that the local Langlands functoriality holds for the pair $(G,\rho)$ at all local places of $k$, where $\psi$ is a non-trivial additive character of $k\bks\BA$. Such
general formulae on $\GL_1$, as a vast generalization of the classical Poisson summation formula,
are expected to be responsible for the Langlands conjecture (\cite{L70}) on global functional equation for the automorphic $L$-functions $L(s,\sig,\rho)$. In order to understand such Poisson summation formulae, we continue with \cite{JL21} and develop a further local theory related to the $(\sig,\rho)$-Schwartz space
$\CS_{\sig,\rho}(\BA^\times)$ and $(\sig,\rho)$-Fourier operator $\CF_{\sig,\rho}$. More precisely,
over any local field $k_\nu$ of $k$, we define distribution kernel functions $k_{\sig_\nu,\rho,\psi_\nu
}(x)$ on $\GL_1$ that represent the $(\sig_\nu,\rho)$-Fourier operators $\CF_{\sig_\nu,\rho,\psi_\nu}$
as convolution integral operators, i.e. generalized Hankel transforms, and the local Langlands $\gam$-functions $\gam(s,\sig_\nu,\rho,\psi_\nu)$ as Mellin transform of the kernel functions.
As a consequence, we show that
any local Langlands $\gam$-functions are the gamma functions in the sense of I. Gelfand, M. Graev, and
I. Piatetski-Shapiro in \cite{GGPS} and of A. Weil in \cite{W66}.
\end{abstract}

\maketitle
\tableofcontents

\section{Introduction}\label{sec-I}

Let $k$ be a number field and $\BA$ be the ring of adeles of $k$. For a $k$-split reductive algebraic group $G$, let $\rho$ be an $n$-dimensional complex representation of its complex dual group $G^\vee(\BC)$.
In our paper (\cite{JL21}), under the assumption of the Langlands conjecture of local functoriality for
$(G,\rho)$ at all local places $\nu$ of $k$, we introduce the $(\sig,\rho)$-Schwartz space
$\CS_{\sig,\rho}(\BA^\times)$ on $\BA^\times$, and the $(\sig,\rho)$-Fourier operator
$\CF_{\sig,\rho,\psi}$, with a non-trivial additive character $\psi$ of $k\bks\BA$, that takes any $(\sig,\rho)$-Schwartz function $\phi$ in $\CS_{\sig,\rho}(\BA^\times)$ to a
$(\wt{\sig},\rho)$-Schwartz function $\CF_{\sig,\rho,\psi}(\phi)$ in $\CS_{\wt{\sig},\rho}(\BA^\times)$, where $\wt{\sig}$ is the contragredient of $\sig$. In \cite[Theorem 5.8]{JL21},
we show that for any given $\sigma\in\CA_\cusp(G)$, the $(\sig,\rho)$-theta function
\[
\Theta_{\sig,\rho}(x,\phi):=\sum_{\alp\in k^\times}\phi(\alp x)
\]
converges absolutely for any $\phi\in\CS_{\sig,\rho}(\BA^\times)$ and $x\in\BA^\times$. The following
is Conjecture 6.4 in \cite{JL21}, which asserts the $(\sig,\rho)$-Poisson summation formula on $\GL_1$.

\begin{cnj}[$(\sig,\rho)$-Poisson Summation Formula]\label{cnj:PSF}
Let $G$ be a $k$-split reductive group, and
$\rho\colon G^\vee(\BC)\to\GL_n(\BC)$ be any finite dimensional representation of the complex
dual group $G^\vee(\BC)$.
For any $\sigma\in\CA_\cusp(G)$, there exist $k^\times$-invariant linear functionals
$\CE_{\sigma,\rho}$ and $\CE_{\wt{\sigma},\rho}$ on $\CS_{\sigma,\rho}(\BA^\times)$ and
$\CS_{\wt{\sigma},\rho}(\BA^\times)$, respectively, such that the
$(\sigma,\rho)$-Poisson Summation Formula:
\begin{align}\label{RPSF}
\CE_{\sigma,\rho}(\phi)
=
\CE_{\wt{\sigma},\rho}(\CF_{\sigma,\rho,\psi}(\phi))
\end{align}
holds for $\phi\in\CS_{\sigma,\rho}(\BA^\times)$. Moreover,
if $\phi\in\CS_{\sig,\rho}^{\circ\circ}(\BA^\times)\subset\CS_{\sig,\rho}(\BA^\times)$, then the linear functional $\CE_{\sig,\rho}(\phi)$ is of the following form
\[
\CE_{\sigma,\rho}(\phi^x)=\Theta_{\sig,\rho}(x,\phi) = \sum_{\alp\in k^\times}\phi(\alp x)
\]
where $\phi^x(\cdot) = \phi(\cdot x)$ with $x\in\BA^\times$.
\end{cnj}
Here $\CS_{\sig,\rho}^{\circ\circ}(\BA^\times)$ is the subspace of $\CS_{\sig,\rho}(\BA^\times)$ spanned by
functions $\phi=\otimes_\nu\phi_\nu$ with compact support conditions at two local places. We refer to
Theorem 6.3 of \cite{JL21} for details.

In J. Tate's thesis (\cite{Tt50}), the classical Poisson summation formula is responsible to the global
functional equation of the Hecke $L$-functions.
The $(\sig,\rho)$-Poisson summation formula on $\GL_1$ in Conjecture \ref{cnj:PSF} is a vast generalization of the classical Poisson summation formula used in J. Tate's thesis
and is expected to be responsible to the Langlands
conjecture of the global functional equation for the Langlands automorphic $L$-function $L(s,\sig,\rho)$.
We prove in Theorem 4.7 of \cite{JL21} that Conjecture \ref{cnj:PSF} is true when $G=\GL_n$,
$\rho$ is the standard representation of $\GL_n(\BC)$, and $\pi$ is an irreducible cuspidal automorphic
representation of $\GL_n(\BA)$. See \cite[Theorem 6.3]{JL21} for a variant when $\pi$ is an irreducible
square-integrable automorphic representation of $\GL_n(\BA)$. It is clear that if the Langlands conjecture
of global functoriality holds for $(G,\rho)$, and the image $\pi$ of an irreducible cuspidal automorphic
representation $\sig$ of $G(\BA)$ under the global functorial transfer associated to $\rho$ is cuspidal, then
the $(\sig,\rho)$-Poisson summation formula on $\GL_1$ holds. The question remains: {\sl How to establish
Conjecture \ref{cnj:PSF} without using the Langlands conjecture of global functoriality for $(G,\rho)$?}

The goal of the paper is to understand further local aspect of the analysis closely related to such Schwratz
functions and Fourier operators over local fields of characteristic zero, as a continuation of the
local theory developed in \cite{JL21}.

Let $F$ be a local field of characteristic zero, and $G$ be an $F$-split reductive group.
Take $\rho\colon G^\vee(\BC)\to\GL_n(\BC)$ to be any
finite dimensional representation of the complex dual group $G^\vee(\BC)$.
We denote by $\Pi_F(G)$ the set of equivalence classes of irreducible admissible smooth representations of
$G(F)$, which are of Casselman-Wallach type if $F$ is an Archimedean local field (\cite{C89}, \cite{Wal92}, and also \cite{SZ11} and \cite{BeK14}).

As explained in \cite[Sections 6.1 and 6.2]{JL21}, we have to fix a local Langlands reciprocity map $\FR_{F,G}$ for $G$ over $F$.
For any Archimedean local fields, the local Langlands conjecture for $G$ is a theorem of
Langlands, which follows from the Langlands classification theory (\cite{L89}). At any non-Archimedean local places, for unramified representations,
their local Langlands parameters are uniquely determined by the Satake isomorphism (\cite{S63} and also \cite{C79}).
The only situation that needs more discussion is that $F$ is non-Archimedean and the representations of $G(F)$ are ramified.

Let $\CW_F$ be the Weil group attached to $F$. The set of local Langlands parameters is denoted by
$\Phi_F(G)$, which consists of continuous, Frobenius semisimple homomorphisms
\begin{align}\label{L-p}
\varsigma\colon \CW_F\times\SL_2(\BC)\longrightarrow G^\vee,
\end{align}
up to conjugation by $G^\vee$. The local Langlands conjecture asserts that there exists a reciprocity map
\begin{align}\label{FR}
\FR_{F,G}\colon \Pi_F(G)\longrightarrow\Phi_F(G),
\end{align}
which is expected to be surjective with finite fibers, and satisfy
a series of compatibility conditions. One of the key issues is to formulate and prove the uniqueness of such local Langlands reciprocity map.

When $G=\GL_n$, it is a theorem of Harris-Taylor (\cite{HT01}), of G. Henniart (\cite{H00}) and of
P. Scholze (\cite{Sc13}) that the local Langlands reciprocity map exists and is unique with compatibility of
local factors, plus other conditions.
However, such a uniqueness is not known in general. When $G$ is an $F$-quasisplit classical group,
such a local Langlands reciprocity map exists due to the endoscopic classification of J. Arthur
(\cite{Ar13}).
In their recent work (\cite{FS21}), L. Fargues and P. Scholze use the geometrization method to understand the local Langlands conjecture. In particular, they establish a local Langlands reciprocity
map for any $F$-split reductive groups considered in this paper. More precise, Theorem I.9.6 of
\cite{FS21} asserts that for any $F$-split reductive group $G$, there exists a local Langlands reciprocity
map $\FR_{F,G}$ from $\Pi_F(G)$ to $\Phi_F(G)$, satisfying nine compatibility conditions. In particular
when $G=\GL_n$, the reciprocity map of Fargues and Scholze coincides with the unique one for $\GL_n$.
When $G$ is an $F$-quasisplit classical group, the reciprocity map of Fargues and Scholze coincides with
the one by Arthur. Although it is still not known (as far as the authors know) if the reciprocity map of Fargues and Scholze is unique, it is the most promising one towards the local Langlands conjecture in
great generality.

From now on, we are going to take the following assumption.
\begin{ass}\label{FarS}
Over any non-Archimedean local field $F$ of characteristic zero, for any $F$-split reductive group $G$,  the reciprocity map $\FR_{F,G}$ exists for the local Langlands conjecture for $G$.
\end{ass}

With Assumption \ref{FarS},
for any $\sig\in\Pi_F(G)$, there exists a unique irreducible admissible representation
\begin{align}\label{sigma-pi}
\pi=\pi(\sigma,\rho),
\end{align}
which belongs to $\Pi_F(\GL_n)$. We may define that
\begin{align}\label{L-gamma}
L(s,\sigma,\rho):=L(s,\pi)\ \
{\rm and}\ \ \gamma(s,\sigma,\rho,\psi):=\gamma(s,\pi,\psi).
\end{align}
Following from \cite[Section 5.3]{JL21}, we define the {\bf $(\sigma,\rho)$-Schwartz space} on
$F^\times$ to be
\begin{align}\label{localSS}
\CS_{\sigma,\rho}(F^\times):=\CS_{\pi}(F^\times).
\end{align}
Here the $\pi$-Schwartz space $\CS_{\pi}(F^\times)$ is defined \cite[Definition 3.3]{JL21}, for
any $\pi\in\Pi_F(\GL_n)$ and will be recalled in \eqref{piSS}.
When $\sigma$ is unramified and $F$ is non-Archimedean, we define the $(\sigma,\rho)$-basic function $\BL_{\sigma,\rho}$ to be the $\pi$-basic function $\BL_{\pi}$, which is given in
\cite[Theorem 3.10]{JL21}. For a fixed non-trivial additive character $\psi$ of $F$, we define the
{\bf $(\sigma,\rho)$-Fourier operator}
$\CF_{\sig,\rho,\psi}$ on $F^\times$ to be
\begin{align}\label{localFO}
\CF_{\sigma,\rho,\psi}:=\CF_{\pi,\psi},
\end{align}
which is a linear transformation from the $(\sigma,\rho)$-Schwartz space
$\CS_{\sigma,\rho}(F^\times)$ to the $(\wt{\sigma},\rho)$-Schwartz space
$\CS_{\wt{\sigma},\rho}(F^\times)$. Here $\wt{\sig}$ is the contragredient of $\sig$. The
$\pi$-Fourier operator $\CF_{\pi,\psi}$ is defined in \cite[Section 3.2]{JL21} for any
$\pi\in\Pi_F(\GL_n)$ and will be recalled in \eqref{eq:1-FO}. Since the image
$\pi=\pi(\sigma,\rho)$ of $\sig$ under the local functorial transfer associated to
$\rho\colon G^\vee(\BC)\to\GL_n(\BC)$ is unique, the $(\sigma,\rho)$-Schwartz space
$\CS_{\sigma,\rho}(F^\times)$ and the $(\sigma,\rho)$-Fourier operator $\CF_{\sig,\rho,\psi}$ are
well defined.

We denote by $\FX(F^\times)$ the set of all quasi-characters of $F^\times$ and may write $\chi_s(a):=\chi(a)|a|_F^s$ for $a\in F^\times$ and $\chi\in\FX(F^\times)$.
The local theory on $\GL_1$ for the pair $(G,\rho)$ can be stated as follows:

\begin{thm}[Local Theory for $(G,\rho)$]\label{thm:LT-Grho}
Let $G$ be a $F$-split reductive group. Take
\[
\rho\colon G^\vee(\BC)\to\GL_n(\BC)
\]
to be any finite dimensional representation of the complex dual group $G^\vee(\BC)$.
Assume that the Langlands conjecture of local functoriality holds for $(G,\rho)$ over $F$.
\begin{enumerate}
\item The local zeta integral defined by
\[
\CZ(s,\phi,\chi):=\int_{F^\times}\phi(x)\chi(x)|x|_F^{s-\frac{1}{2}}\ud^\times x
\]
for $\phi\in\CS_{\sigma,\rho}(F^\times)$ and $\chi\in\FX(F^\times)$ converges absolutely for $\Re(s)$
sufficiently positive, admits a meromorphic continuation to $s\in\BC$, and satisties the functional equation
\[
\CZ(1-s,\CF_{\sig,\rho,\psi}(\phi),\chi^{-1})
=
\gamma(s,\sigma\times\chi,\rho,\psi)\cdot\CZ(s,\phi,\chi).
\]
\item $\CZ(s,\phi,\chi)$ is a holomorphic multiple of the Langlands local $L$-function
$L(s,\sig\times\chi,\rho)$ associated to $(\sig,\chi)$  and $\rho$.
Moreover, when $F$ is non-Archimedean, the fractional ideal generated by the local zeta integrals $\CZ(s,\phi,\chi)$ is of the form:
\begin{align*}
\{\CZ(s,\phi,\chi)\mid \phi\in \CS_\pi(F^\times) \}
=L(s,\sig\times\chi,\rho)\cdot \BC[q^s,q^{-s}];
\end{align*}
and when $F$ is Archimedean, $\CZ(s,\phi,\chi)$, with unitary characters $\chi$, has the following property: Over any vertical strip for any $a<b$:
\[
S_{a,b}:=\{s\in\BC\mid a\leq\Re(s)\leq b\},
\]
if $P_\chi(s)$ is a polynomial in $s$ such that the product $P_\chi(s)L(s,\pi\times\chi)$ is bounded in the vertical strip $S_{a,b}$, with small neighborhoods at the possible poles of the $L$-function
$L(s,\sig\times\chi,\rho)$ removed, then the product $P_\chi(s) \CZ(s,\phi,\chi)$ must be bounded in the same vertical strip $S_{a,b}$, with small neighborhoods at the possible poles of the $L$-function
$L(s,\sig\times\chi,\rho)$ removed.
\item When $F$ is non-Archimedean, and $\sig$ is unramified, there exists a $(\sig,\rho)$-basic function
$\BL_{\sig,\rho}(x)$, belonging to $\CS_{\sigma,\rho}(F^\times)$, such that the following identity
\[
\CZ(s,\BL_{\sig,\rho},\chi)=L(s,\sig\times\chi,\rho)
\]
holds for any unramified characters $\chi$ and all $s\in\BC$ as meromorphic functions in $s$; and
$\CF_{\sig,\rho,\psi}(\BL_{\sig,\rho})=\BL_{\wt{\sig},\rho}$.
\item Define the $(\sig,\rho)$-kernel function
\[
k_{\sig,\rho,\psi}(x):=k_{\pi,\psi}(x)
\]
where $\pi=\pi(\sig,\rho)$ is the local Langlands functorial image of $\sig$ associated to $(G,\rho)$, and
the $\pi$-kernel function $k_{\pi,\psi}(x)$ is given in Proposition \ref{k-smooth} and
Corollary \ref{cor:kernel}. Then as distributions on $F^\times$,
\[
\CF_{\sig,\rho,\psi}(\phi_0)(x)=(k_{\sig,\rho,\psi}*\phi_0^\vee)(x)
\]
for any $\phi_0\in\CC_c^\infty(F^\times)$.
\item The Langlands local $\gam$-functions $\gamma(s,\sigma\times\chi,\rho,\psi)$ with any
unitary characters $\chi$ of $F^\times$ are the gamma functions in the sense of I. Gelfand, M. Graev and
I. Piatetski-Shapiro in \cite{GGPS} and of A. Weil in \cite{W66}, i.e.
\[
\CF_{\sig,\rho,\psi}(\chi_s^{-1})=\gamma(\frac{1}{2},\sigma\times\chi_s,\rho,\psi)\cdot\chi_s,
\]
as distributions on $F^\times$. The identity holds for $s\in\BC$ after meromorphic continuation, where the Fourier
transform $\CF_{\sig,\rho,\psi}(\chi_s)$ is defined by
\[
(\CF_{\sig,\rho,\psi}(\chi_s),\phi_0):=(\chi_s,\CF_{\sig,\rho,\psi}(\phi_0))
\]
for any $\phi_0\in\CC_c^\infty(F^\times)$. Moreover, the following identity
\[
\CF_{\sig,\rho,\psi}(\chi_s)(x)=(k_{\sig,\rho,\psi}*\chi_s^\vee)(x)
\]
holds for $s\in\BC$ after meromorphic continuation, as distributions on $F^\times$, where the
convolution $k_{\sig,\rho,\psi}*\chi_s^\vee$ is given in Theorem \ref{GGPS-gam}.
\end{enumerate}
\end{thm}

It is clear that the above local theory depends on the local Langlands reciprocity map $\FR_{F,G}$. Before knowing the uniqueness of the local Langlands reciprocity map, we may only
obtain the corresponding global theory for partial $L$-functions. However, this is not an issue for the local theory in this paper.

With Assumption \ref{FarS},
it is enough to prove Theorem \ref{thm:LT-Grho} for $G=\GL_n$ and $\pi=\pi(\sig,\rho)$. In this sense,
Parts (1), (2) and (3) of Theorem \ref{thm:LT-Grho} is a reformulation of Theorems 3.4 and 3.10 of \cite{JL21}. In this paper, we will focus on the proof of Parts (4) and (5) of Theorem \ref{thm:LT-Grho}
for $G=\GL_n$ and $\pi\in\Pi_F(\GL_n)$.

After we review the local theory on $\GL_1$ for the standard $L$-function of $\GL_n$ in
Section \ref{sec-SSFO} as developed mainly in \cite[Section 3]{JL21}, which can be traced back to the
classical work of R. Godement and H. Jacquet (\cite{GJ72}), we start our investigation
in Section \ref{sec-KF-na} on the $\pi$-kernal function $k_{\pi,\psi}(x)$ on $F^\times$ for each
$\pi\in\Pi_F(\GL_n)$ when $F$ is non-Archimedean. We prove some technical lemmas in
Section \ref{ssec-Lemmas}, which are needed for the definition of the $\pi$-kernel function
\begin{align}\label{kpi-na}
k_{\pi,\psi}(x) :=
\int^\reg_{\det g=x}
\Phi_{\GJ}(g)\vphi_{\wt{\pi}}(g)\ud_x g
=
\lim_{\ell\to \infty}
\int_{\det g=x}
\left(\Phi_{\GJ}*\Fc_{\ell}^\vee
\right)
(g)\vphi_{\wt{\pi}}(g)\ud_xg
\end{align}
as in \eqref{kernel-na} and \eqref{kernel-pv}.
The generalized function $\Phi_{\GJ}$ is the kernel function for the local theory of Godement-Jacquet, which is renormalized in \cite[Section 2.3]{JL21} and is recalled in \eqref{GJ-kernel}. Here $\wt{\pi}$ is the contragredient of $\pi$ and $\vphi_{\wt{\pi}}(g)\in\CC(\wt{\pi})$ which is the space of all matrix coefficients of
$\wt{\pi}$ spanned by functions $\{g\mapsto \langle \wt{\pi}(g)\wt{v},v\rangle\mid v\in \pi,\wt{v}\in \wt{\pi}\}$.
We refer to \eqref{kernel-na} and \eqref{kernel-pv} for unexplained notations in \eqref{kpi-na}. Proposition \ref{k-smooth} shows that
the limit in \eqref{kpi-na} converges absolutely and the convergence is uniform when $x$ lies in a compact neighborhood. Hence the $\pi$-kernel function $k_{\pi,\psi}(x)$ is smooth on $F^\times$.
In order to understand the $\pi$-kernel function $k_{\pi,\psi}(x)$, we study the $\chi_s$-Fourier coefficient
of $k_{\pi,\psi}(x)$, which is defined by
\begin{align}\label{kchis-na}
\int^\pv_{F^\times}
k_{\pi,\psi}(x)
\chi_s(x^{-1})\ud^\times x
=
\lim_{\ell\to \infty}\sum_{m=-\ell}^\ell
\int_{\RS_m}
k_{\pi,\psi}(x)
\chi_s(x^{-1})\ud^\times x,
\end{align}
for any quasi-characters $\chi_s$ of $F^\times$, where $\RS_m:=\{x\in F^\times\mid\ |x|_F=q^{-m}\}$.
Theorem \ref{thm:KF-na} shows that the principal value integral in \eqref{kchis-na}
is convergent for $\Re(s)$ sufficiently small and admits a meromorphic continuation to $s\in \BC$ with the following identity:
\begin{align}\label{kchis-gamma}
\int^\pv_{F^\times}
k_{\pi,\psi}(x)
\chi_s(x^{-1})\ud^\times x
=
\gam(\frac{1}{2},\pi\times \chi_s,\psi).
\end{align}
In consequence, we prove in Corollary \ref{kernel-wd} that the $\pi$-kernel function $k_{\pi,\psi}(x)$
in \eqref{kpi-na} is well-defined in the sense that it is independent of the choice of the matrix coefficient
$\vphi_{\wt{\pi}}$ of $\wt{\pi}$ and the choice of the sequence $\{\Fc_\ell\}_{\ell\geq 1}$.

The Archimedean counterpart of Section \ref{sec-KF-na} is given in Section \ref{sec-KF-ar}.
The $\pi$-kernel function $k_{\pi,\psi}(x)$ is defined in \eqref{kernel-ar}, similar to that in \eqref{kpi-na}.
Corollary \ref{cor:kernel} shows that $k_{\pi,\psi}(x)$ is smooth on $F^\times$.
For any quasi-characters $\chi_s$, the $\chi_s$-Fourier coefficients of $k_{\pi,\psi}(x)$ and their properties
are established in Theorem \ref{thm:KF-ar}. It is worthwhile to point out that the regularization to define the
$\pi$-kernel functions $k_{\pi,\psi}(x)$ is technical in both non-Archimedean and Archimedean cases. The
proof in the Archimedean case (Proposition \ref{prp:k-ar}) uses the elliptic regularity of the Casimir operator
$\Delta$ on $\GL_n(F)$ (\cite[Lemma 3.7]{BeK14}).

Finally, in Section \ref{sec-FOHT}, we take $F$ to be any local field of characteristic zero and prove in
Theorem \ref{thm:HT} that for any $\pi\in\Pi_F(\GL_n)$, the Fourier operator $\CF_{\pi,\psi}$ can be represented as a Hankel transform (convolution operator) with the kernel function $k_{\pi,\psi}(x)$
\[
\CF_{\pi,\psi}(\phi_0)(x)=(k_{\pi,\psi}*\phi_0^\vee)(x)
\]
for any $\phi_0\in\CC_c^\infty(F^\times)$. If any quasi-character $\chi_s$ of $F^\times$ is regarded as
a homogeneous distribution (generalized function) on $F^\times$, we prove in Theorem \ref{GGPS-gam}
that
\[
\CF_{\pi,\psi}(\chi_s^{-1})=\gam(\frac{1}{2},\pi\times\chi_s,\psi)\cdot\chi_s
\]
and
\[
\CF_{\pi,\psi}(\chi_s)(x)=(k_{\pi,\psi}*\chi_s^\vee)(x),
\]
hold for $s\in\BC$ after meromorphic continuation. Hence for any $\pi\in\Pi_F(n)$, the local Langlands $\gam$-functions $\gam(s,\pi\times\chi,\psi)$ with any
unitary characters $\chi$ of $F^\times$ are the gamma functions in the sense of Gelfand, Graev and
Piatetski-Shapiro in \cite{GGPS} and of Weil in \cite{W66}. This completes the proof of Theorem \ref{thm:LT-Grho}.

Comparing with the local theory of the Braverman-Kazhdan proposal for Langlands automorphic $L$-functions
(\cite{BK00}), the approach we take in \cite{JL21} and this paper has the advantage that the local theory can be
completely established based on the classical work of Godement-Jacquet (\cite{GJ72}) and Assumption \ref{FarS}. The existence and uniqueness of the local Langlands reciprocity map is known
for many important cases. Of course, the key point is the global theory, i.e. the
corresponding Poisson summation formula. In the global theory of the Braverman-Kazhdan proposal (\cite{BK00}) and the formulation of B. C. Ng\^o (\cite{N20}) one has to face the singularities of the
reductive monoid associated to the pair $(G,\rho)$, while in our approach, we have to face the difficulties in
analysis in order to understand the $(\sig,\rho)$-Schwartz space $\CF_{\sig,\rho}(\BA^\times)$ and the
$(\sig,\rho)$-Fourier operator $\CF_{\sig,\rho,\psi}$ on $\GL_1(\BA)$, in order to establish
Conjecture \ref{cnj:PSF}, the $(\sig,\rho)$-Poisson summation formula on $\GL_1$.

We would like to thank Binyong Sun and Chengbo Zhu for helpful discussions related to the work in \cite{BeK14}, and thank
the referee for helpful comments, which include a mistake in the proof of Proposition \ref{prp:k-ar} in an early version of this paper.

\section{$\pi$-Schwartz Spaces and $\pi$-Fourier Operators}\label{sec-SSFO}

Let $F$ be a local field of characteristic zero. If $F$ is non-Archimedean, we denote by $\Fo_F$
the ring of the integers on $F$ and by $\Fp = \Fp_F$ the maximal ideal of $\Fo=\Fo_F$.
Let $\RG_n:=\GL_n$ be the general linear group defined over $F$. Fix the following maximal (open if $F$ is non-Archimedean) compact subgroup $K$ of $\RG_n(F)=\GL_n(F)$,
\begin{align}\label{K}
K=
\begin{cases}
\GL_n(\Fo_F), & F\text{ is non-Archimedean};\\
\RO(n), &  F=\BR;\\
\RU(n), & F=\BC.
\end{cases}
\end{align}
Fix the Haar measure $\ud g = \frac{\ud^+g}{|\det g|_F^n}$ on $\RG_n(F)$ where $\ud^+g$ is the measure on $\RM_n(F)$, the space of all $n\times n$ matrices over $F$, induced from the standard additive measure $\ud^+ x$ on $F$ that is in particular self-dual with respect to a given additive character $\psi=\psi_F$ as
fixed in \cite[\S2.1]{JL21}. As usual, $\RG_n(F)$ embeds into $\RM_n(F)$ in a standard way.

Let $\Pi_F(n)$ be the set of equivalence classes of irreducible smooth representations of $\RG_n(F)$ when $F$ is non-Archimedean;
and of irreducible Casselman-Wallach representations of $\RG_n(F)$ when $F$ is Archimedean.
Set $\CC(\pi)$ the space of smooth matrix coefficients attached to $\pi$.

Let $\CS(\RM_n(F))$ be the space of Schwartz-Bruhat functions on $\RM_n(F)$. In the Archimedean case, it is the space of usual Schwartz functions. The standard Fourier transform $\CF_\psi$ acting on $\CS(\RM_n(F))$ is defined as follows,
\begin{equation}\label{eq:FTMAT}
\CF_\psi(f)(x) = \int_{\RM_n(F)}
\psi(\tr(xy))f(y)\ud^+y.
\end{equation}
Moreover,  the standard Fourier transform $\CF_\psi$ extends to a unitary operator on the space $L^2(\RM_n(F),\ud^+x)$ and satisfies the following identity:
\begin{equation}\label{eq:FTId}
\CF_\psi\circ \CF_{\psi^{-1}}  =\Id.
\end{equation}
For $f\in\CS(\RM_n(F))$, we define
\begin{align}\label{GJ-SF}
\xi_f(g):=|\det g|_F^{\frac{n}{2}}\cdot f(g)
\end{align}
for $g\in\RG_n(F)$. Then we define the Schwartz space on $\RG_n(F)$ to be
\begin{align}\label{GJ-SS}
\CS_{\std}(\RG_n(F)):=\{\xi\in\CC^\infty(\RG_n(F))\mid \quad |\det g|^{-\frac{n}{2}}\cdot\xi(g)\in\CS(\RM_n(F))\}.
\end{align}
By \cite[Prposition 2.5]{JL21}, the Schwartz space $\CS_{\std}(\RG_n(F))$ is a subspace of $L^2(\RG_n(F),\ud g)$, which is the space of square-integrable functions on $\RG_n(F)$.

Following the reformulation of the local theory of Godement-Jacquet in \cite[Section 2.3]{JL21}, the distribution kernel is
\begin{align}\label{GJ-kernel}
\Phi_{\GJ}(g):=\psi(\tr g)\cdot|\det g|_F^{\frac{n}{2}},
\end{align}
and the Fourier operator $\CF_{\GJ}$ is
\begin{align}\label{GJ-FO}
\CF_{\GJ}(\xi)(g):=\left(\Phi_{\GJ}*\xi^\vee\right)(g)
\end{align}
for any $\xi\in\CS_{\std}(\RG_n(F))$. From \cite[Proposition 2.6]{JL21}, we obtain a relation between
the Fourier operator $\CF_{\GJ}$ and the classical Fourier transform $\CF_\psi$:
\begin{align}\label{FOFT}
\CF_{\GJ}(\xi)(g)=\left(\Phi_{\GJ}*\xi^\vee\right)(g)
=
|\det g|_F^{\frac{n}{2}}\cdot\CF_\psi(|\det g|^{-\frac{n}{2}}\xi)(g).
\end{align}
From the proof of \cite[Proposition 2.6]{JL21}, it is easy to obtain that
\begin{align}\label{convolution}
\left(\Phi_{\GJ}*\xi^\vee\right)(g)
=
|\det g|_F^{\frac{n}{2}}\left(\psi(\tr(\cdot))*(|\det(\cdot)|_F^{\frac{n}{2}}\xi)^\vee\right)(g)
\end{align}
for any $\xi\in\CS_{\std}(\RG_n(F))$.

We write, for any $\xi\in\CS_{\std}(\RG_n(F))$, $\xi(g)=|\det g|_F^{\frac{n}{2}}\cdot f(g)$ for some $f\in\CS(\RM_n(F))$. The local zeta integral of Godement-Jacquet can be renormalized as
\begin{align}\label{GJ-Zeta}
 \CZ(s,\xi,\vphi_\pi,\chi)=
\int_{\RG_n(F)}
\xi(g)\vphi_\pi(g)\chi(\det g)|\det g|_F^{s-\frac{1}{2}}\ud g.
\end{align}
It converges absolutely for $\Re(s)$ sufficiently positive and
satisfies the functional equation:
\begin{align}\label{GJFE}
\CZ(1-s, \CF_{\GJ}(\xi), \vphi_\pi^\vee,\chi^{-1})
=
\gam(s,\pi\times\chi,\psi)\cdot
\CZ(s,\xi,\vphi_\pi,\chi),
\end{align}
which holds as meromorphic functions in $s$, where $\gam(s,\pi\times\chi,\psi)$ is the local
Langlands $\gam$-function associated to $\pi\in\Pi_F(n)$ and $\chi\in\FX(F^\times)$, the set of all quasi-characters of $F^\times$.

\subsection{$\pi$-Schwartz functions}\label{ssec-piSF}
Consider the following determinant map:
\begin{align}\label{det}
\det={\det}_F\colon\RG_n(F)=\GL_n(F)\to F^\times.
\end{align}
It is clear that the kernel $\ker(\det)=\SL_n(F)$. For each $x\in F^\times$, the fiber of the determinant map
$\det$ is
\begin{align}\label{fiber}
\RG_n(F)_x:=
\{g\in \RG_n(F)\mid \det g=x\}.
\end{align}
It is clear that each fiber $\RG_n(F)_x$ is an $\SL_n(F)$-torsor. Hence one has the $\SL_n(F)$-invariant measure $\ud_x g$ that is induced from the Haar measure $\ud_1 g$
on $\SL_n(F)$. Here $\ud_1 g$ is the restriction of the Haar measure $\ud g$ from $\RG_n(F)$ to $\SL_n(F)$.

Write $\xi=|\det g|^{\frac{n}{2}}\cdot f(g)\in\CS_{\std}(\RG_n(F))$ with some $f\in\CS(\RM_n(F))$
as defined in \eqref{GJ-SS}. For $\pi\in\Pi_F(n)$,  we denote by $\CC(\pi)$ the space of all matrix coefficients of $\pi$.  For $\varphi_\pi\in\CC(\pi)$,  as in \cite[Section 3.1]{JL21},
we define
\begin{align}\label{fibration}
\phi_{\xi,\vphi_\pi}(x) := \int_{\RG_n(F)_x}
\xi(g)\vphi_\pi(g)\ud_x g
=
|x|_F^{\frac{n}{2}}
\int_{\RG_n(F)_x}
f(g)\vphi_\pi(g)\ud_x g.
\end{align}
By \cite[Proposition 3.2]{JL21}, the function $\phi_{\xi,\vphi_\pi}(x)$ is absolutely convergent for all $x\in F^\times$ and is smooth over $F^\times$. Following \cite[Definition 3.3]{JL21},
for any $\pi\in\Pi_F(n)$, the space of $\pi$-Schwartz functions is defined by
\begin{align}\label{piSS}
\CS_\pi(F^\times) = \Span
\{
\phi=\phi_{\xi,\vphi_\pi}\in\CC^\infty(F^\times)\mid \xi\in \CS_{\std}(\RG_n(F)),\vphi_\pi\in \CC(\pi)\}.
\end{align}
By \cite[Corollary 3.8]{JL21}, we have
\begin{align}\label{CSC}
\CC_c^\infty(F^\times)\subset
\CS_\pi(F^\times)
\subset
\CC^\infty(F^\times).
\end{align}
As in \cite[Section 3.1]{JL21}, for any $\phi\in \CS_\pi(F^\times)$ and a quasi-character $\chi\in\FX(F^\times)$, define a $\GL_1$ zeta integral $\CZ(s,\phi,\chi)$ associated to the pair $(\phi,\chi)$ to be
\begin{equation}\label{eq:1-zeta}
\CZ(s,\phi,\chi) =
\int_{F^\times}
\phi(x)\chi(x)|x|_F^{s-\frac{1}{2}}\ud^\times x.
\end{equation}
When $\phi=\phi_{\xi,\vphi_\pi}$ for some $\xi\in\CS_{\std}(\RG_n(F))$ and
$\varphi_\pi\in\CC(\pi)$, we have the following identity of local zeta integrals:
\begin{align}\label{eq:zetas}
\CZ(s,\phi,\chi)=\CZ(s,\xi,\vphi_\pi,\chi),
\end{align}
which holds for $\Re(s)$ sufficiently large and then for all $s\in\BC$ by meromorphic continuation.

\subsection{$\pi$-Fourier operators}\label{ssec:FO-GL1}
As in \cite[Section 3.2]{JL21}, for $\phi\in\CS_\pi(F^\times)$, we define the Fourier operator $\CF_{\pi,\psi}(\phi)$ through the following diagram:
\begin{align}\label{diag:F}
\xymatrix{
\CS_{\std}(\RG_n(F))\otimes \CC(\pi)\ar[d]\ar[rrr]^{(\CF_{\GJ},(\cdot)^{\vee})}&&& \CS_{\std}(\RG_n(F))\otimes \CC(\wt{\pi})\ar[d]\\
\CS_\pi(F^\times) \ar[rrr]^{\CF_{\pi,\psi}} &&& \CS_{\wt{\pi}}(F^\times)
}
\end{align}
More precisely, for $\phi=\phi_{\xi,\vphi_\pi}\in\CS_\pi(F^\times)$ with a $\xi\in\CS_{\std}(\RG_n(F))$ and
a $\vphi_\pi\in\CC(\pi)$, we define
\begin{align}\label{eq:1-FO}
\CF_{\pi,\psi}(\phi)=\CF_{\pi,\psi}(\phi_{\xi,\vphi_\pi}):=\phi_{\CF_{\GJ}(\xi),\vphi_\pi^\vee},
\end{align}
where $\vphi_\pi^\vee(g)=\vphi_\pi(g^{-1})\in\CC(\wt{\pi})$. By \cite[Proposition 3.9]{JL21},
the $\pi$-Fourier operator is well defined and yields the following functional equation:
\begin{align}\label{GL1-FE}
\CZ(1-s,\CF_{\pi,\psi}(\phi),\chi^{-1}) =
\gam(s,\pi\times \chi,\psi)
\cdot \CZ(s,\phi,\chi),
\end{align}
holds for any $\phi\in \CS_\pi(F^\times)$, after meromorphic continuation, according to \cite[Theorem 3.10]{JL21}.

\section{$\pi$-Kernel Function: non-Archimedean case}
\label{sec-KF-na}

Assume that the local field $F$ is non-Archimedean. We are going to introduce a $\pi$-kernel
function
$k_{\pi,\psi}(x)$ on $F^\times$ for each $\pi\in\Pi_F(n)$, such that the local Langlands
$\gam$-function $\gam(s,\pi,\psi)$ can be represented by the Mellin transform of
$k_{\pi,\psi}(x)$. Meanwhile, we show that $\gam(s,\pi,\psi)$ is a gamma function in the
sense of Gelfand, Graev, and Piatetski-Shapiro in \cite{GGPS} and of Weil in \cite{W66}.

\subsection{Some technical lemmas}\label{ssec-Lemmas}
For any integer $\ell\geq 1$ and $\Fp=\Fp_F$,
we define $K_\ell = \RI_n+\RM_n(\Fp^\ell)$, the principal congruence (module $\Fp^\ell$) subgroup of the maximal open compact subgroup
$K=\RG_n(\Fo_F)$ of $\RG_n(F)$. The set $\{K_\ell\}_{\ell=1}^\infty$ forms a family of open compact neighborhoods of $\RI_n$ in $\RG_n(F)$.
We establish some technical lemmas, which are needed for the proof of the main result of this section.
Recall that the additive character $\psi=\psi_F$of $F$ is taken as in \cite[Section 2.1]{JL21}, which is precisely defined as follows. Let $\vpi$ be a fixed uniformizer in $\Fp=\Fp_F$.
Then $\psi$ is an additive character of $F$ that is trivial on $\Fo_F$, but non-trivial on $\vpi^{-1}\cdot\Fo_F$.

\begin{lem}\label{lem:zero-na}
Assume that a positive integer $n\geq 2$ and the additive character $\psi=\psi_F$ is given as above.
Let $\Fn$ be any open compact subset of $F^\times$, and $\ell_{\Fn}$ be a sufficiently large integer. For any $g\in \RG_n(F)$ with $\det g\in \Fn$ and any $\ell>m>\ell_\Fn$, assume
$g\in \RM_n(\ell,m):=\RM_n(\Fp^{-\ell})\bs\RM_n(\Fp^{-m})$, then
the following identity holds
$$
\int_{h\in K^{1}_{\ell_\Fn}}
\psi(\tr(gh))\ud h = 0
$$
where $K^{1}_{\ell_\Fn}:=K_{\ell_\Fn}\cap\SL_n(F)$ is an open compact subgroup of $\SL_n(F)$.
\end{lem}

\begin{proof}
Up to a finite union, there will be no harm to assume that $\Fn=\vpi^\alp\cdot \Fo_F^\times$ for some integer $\alp\in \BZ$, where $\vpi$ is the fixed uniformizer in $\Fo_F$.
This implies that $g$ belongs to both $\RG_n(F)_{\Fn} = \{g\in \RG_n(F)\mid \det g\in \vpi^\alp\cdot \Fo^\times_F\}$ and $\RM_n(\ell,m)$.

Consider the Cartan decomposition of $\RG_n(F)$ with respect to the maximal open compact subgroup $K=\RG_n(\Fo_F)$, and write
$$
g =k_1tk_2
$$
with $k_1,k_2\in K$, and $t= \diag(t_1,...,t_n)$ with $|t_1|_F\geq |t_2|_F\geq ...\geq |t_n|_F$. Since the set $\RM_n(\ell,m)=\RM_n(\Fp^{-\ell})\bs\RM_n(\Fp^{-m})$ is stable under the left and right
translations by $K$, we obtain that $t\in \RM_n(\ell,m)$ and $\det t\in \vpi^\alp\cdot \Fo^\times_F$. In particular $t_1\in \Fp^{-\ell}\bs \Fp^{-m}$ and $t_n\in \Fp^{\floor{\frac{\alp+m}{n-1}}}$. Here $\floor{x}$ is the largest integer smaller than or equal to $x$.

In the following we may fix any $\ell_\Fn>n|\alp|$ with $\ell>m>\ell_\Fn$.

Since $K^{1}_{\ell_0}$ is a normal subgroup of $K=\GL_n(\Fo_F)$, after changing variable $h\to k_2^{-1}hk_2$, we can write
\begin{align}\label{eq:CD}
\int_{h\in K^{1}_{\ell_0}}
\psi(\tr(gh))\ud h
=
\int_{h\in K^{1}_{\ell_0}}
\psi(\tr(kth))\ud h =
\int_{h\in K^{1}_{\ell_0}}
\psi(\tr(thk))\ud h
\end{align}
with $k = k_2k_1$.
Since $k\in K$ and $\det k\in\Fo_F^\times$, there exists a permutation $\sig$ of $\{1,...,n\}$, such that the $(\sig(i),i)$-th entry of $k$,
$k_{\sig(i),i}$ belongs to $\Fo^\times_F$, for $i=1,2,\cdots,n$. For such a $\sig$, we are going to construct a compact subgroup $\Theta_\sig$ of $K^1_{\ell_0}$ such that
for the $t$ and $k$ that are determined by $g\in\RG_n(F)_\Fn\cap\RM_n(\ell,m)$ as above and
for any $h\in K^1_{\ell_0}$,
\begin{align}\label{eq:Theta0}
\int_{\Theta_\sig}
\psi(\tr(t\theta hk))\ud \theta = 0
\end{align}
whenever $\ell>m\geq \ell_\Fn$ is sufficiently large. In fact, the vanishing of the integral in \eqref{eq:Theta0} with the reduction in \eqref{eq:CD}
is sufficient for the lemma. More precisely, we consider the following change of variable for $h$:
for any $\theta\in \Theta_\sig$,
$$
\int_{h\in K^1_{\ell_0}}
\psi(\tr(thk))\ud h
=\int_{h\in K^1_{\ell_0}}
\psi(\tr(t\theta hk))\ud h.
$$
By integrating on the variable $\theta$ over the compact subgroup $\Theta_\sig$ on both sides, we obtain that
\[
\vol(\Theta_\sig)\cdot\int_{h\in K^1_{\ell_0}}
\psi(\tr(thk))\ud h
=\int_{\Theta_\sig}\int_{h\in K^1_{\ell_0}}
\psi(\tr(t\theta hk))\ud h\ud\theta.
\]
Since the double integration on the right-hand side is over compact subgroups, we are able to change the order of integrations and obtain that
\[
\vol(\Theta_\sig)\cdot\int_{h\in K^1_{\ell_0}}
\psi(\tr(thk))\ud h=
\int_{h\in K^1_{\ell_0}}
\bigg(\int_{\Theta_\sig}\psi(\tr(thk))\ud\theta\bigg)\ud h.
\]
Hence the vanishing of the integral in \eqref{eq:Theta0} with the reduction in \eqref{eq:CD} implies the lemma.

Now we are going to prove the vanishing of the integral in \eqref{eq:Theta0}, for the $t$ and $k$ that are determined by $g\in\RG_n(F)_\Fn\cap\RM_n(\ell,m)$ as above and
for any $h\in K^1_{\ell_0}$.

First, we consider the case that $\sig(1) = 1$ and $\sig(n) = n$ (with $n\geq 2$). In this case,  we look at the compact subgroup $\Theta_\sig$ in $K^1_{\ell_0}$ consisting of diagonal elements of the form:
\[
\theta=\diag(a,1,\cdots,1,a^{-1})\in K^1_{\ell_0}
\]
with $a\in1+\Fp^{\ell_0}$. Since $h\in K^1_{\ell_0}$ is congruence to $\RI_n$ modulo $\Fp^{\ell_0}$, it is straightforward to see that for the $k$ that is determined by $g$ as above,
we still have
$$
(hk)_{\sig(i),i}\in \Fo^\times_F
$$
for every $h\in K^1_{\ell_0}$. Then the trace $\tr(t\theta hk)$ can be calculated as follows:
\[
\tr(t\theta hk) = t_1 a(hk)_{1,1}
+\sum_{j=2}^{n-1} t_j(hk)_{j,j}
+t_n a^{-1}(hk)_{n,n},
\]
with $a\in1+\Fp^{\ell_0}$. Since $\sig(n)=n$, we have that $(hk)_{n,n}\in\Fo_F^\times$.  Since $t_n\in\Fp^{\floor{\frac{\alp+m}{n-1}}}$ with $\ell>m>n|\alp|$ large, we have $\psi(t_n a^{-1}(hk)_{n,n})=1$.
It follows that
\[
\psi(\tr(t\theta hk))=\psi(t_1 a(hk)_{1,1})\psi(\sum_{j=2}^{n-1} t_j(hk)_{j,j}).
\]
Hence we obtain that the integral in \eqref{eq:Theta0} can be written as
\[
\int_{\Theta_\sig}
\psi(\tr(t\theta hk))\ud \theta=\psi(\sum_{j=2}^{n-1} t_j(hk)_{j,j})\int_{1+\Fp^{\ell_0}}\psi(t_1 a(hk)_{1,1})\ud a.
\]
Since $(hk)_{1,1}\in\Fo_F^\times$ and $t_1\in\Fp^{-\ell}\bs \Fp^{-m}$ with $\ell>m$ sufficiently large, if we write $a=1+x$, then we have
\[
\int_{1+\Fp^{\ell_0}}\psi(t_1 a(hk)_{1,1})\ud a=\psi(t_1(hk)_{1,1})\int_{\Fp^{\ell_0}}\psi(t_1 x(hk)_{1,1})\ud x.
\]
The integral on the right-hand side is zero as long as $\ell>m>\ell_0$, due to the choice of $\psi=\psi_F$. Hence the vanishing of the integral in \eqref{eq:Theta0} is proved in this case.

Next, we consider that case where $\sig(1)=n\neq 1$ and $n\geq 2$. In this case, we take the following compact subgroup $\Theta_{\sig}$ of $K^1_{\ell_0}$ that consists of elements of the form:
$$
\theta=\begin{pmatrix}1&0&b\\0&\RI_{n-2}&0\\c&0&d\end{pmatrix}
\in K^1_{\ell_0}.
$$
It is clear that $b,c\in\Fp^{\ell_0}$ and $d\in 1+\Fp^{\ell_0}$ with $d-bc=1$.
The trace $\tr(t\theta hk)$ can be calculated as follows:
\[
\tr(t\theta hk) = t_1(hk)_{1,1}+t_1b(hk)_{n,1}+
\sum_{j=2}^{n-1}
t_j(hk)_{j,j}
+t_n c (hk)_{1,n}
+t_nd (hk)_{n,n}.
\]
Since $t_n\in\Fp^{\floor{\frac{\alp+m}{2}}}$, $c\in\Fp^{\ell_0}$, $d\in 1+\Fp^{\ell_0}$, and $(hk)_{1,n}, (hk)_{n,n}\in\Fo_F$, we must have that
\[
\psi(t_n c (hk)_{1,n})=\psi(t_nd (hk)_{n,n})=1,
\]
with $\ell>\ell_0$ sufficiently large. The integral in \eqref{eq:Theta0} can be written as
\[
\int_{\Theta_\sig}
\psi(\tr(t\theta hk))\ud \theta
=
\psi(\sum_{j=1}^{n-1}t_j(hk)_{j,j})\int_{\Theta_\sig}\psi(t_1b(hk)_{n,1})\ud\theta.
\]
Since $d-bc=1$, we can write
\[
\begin{pmatrix}1&b\\c&d\end{pmatrix}=\begin{pmatrix}1&0\\c&1\end{pmatrix}\begin{pmatrix}1&b\\0&1\end{pmatrix}.
\]
Hence we obtain that
\[
\int_{\Theta_\sig}\psi(t_1b(hk)_{n,1})\ud\theta
=
\beta\cdot\int_{\Fp^{\ell_0}}\psi(t_1b(hk)_{n,1})\ud b,
\]
which is zero as $(hk)_{n,1}=(hk)_{\sig(1),1}\in\Fo_F^\times$ and $t_1\in\Fp^\ell\bs\Fp^m $ with $\ell>m>\ell_0$, due to the choice of $\psi=\psi_F$. The vanishing of the integral in \eqref{eq:Theta0} is proved in this case. Here the
constant $\beta$ is determined by the volume of the subgroup that consists of elements of type $\begin{pmatrix}1&0\\c&1\end{pmatrix}$ with $c\in\Fp^{\ell_0}$.

It is clear that the same proof can be applied to the case where $\sig(n)=1\neq n$ and $n\geq 2$.

Finally, we consider the case where $\sig(1)=j_0\notin \{1,n\}$ and $n\geq 2$. In this case, we take a compact subgroup $\Theta_\sig$ of $K^1_{\ell_0}$ that consists of elements of the form:
\[
\theta=\begin{pmatrix}1&0&b&0&0\\0&\RI_{j_0-2}&0&0&0\\c&0&1&0&0\\0&0&0&\RI_{n-j_0-1}&0\\0&0&0&0&d
\end{pmatrix}
\]
with $b,c\in\Fp^{\ell_0}$, $d\in1+\Fp^{\ell_0}$ and $(1-bc)d=1$. Hence we obtain that $d=(1-bc)^{-1}\in1+\Fp^{2\ell_0}$.
The trace can be written as
\[
\tr(t\theta hk) =
\sum_{j=1}^{n-1}
t_j(hk)_{j,j}
+t_n d (hk)_{n,n}
+t_1 b
(hk)_{j_0,1}
+t_{j_0}c(hk)_{1,j_0}.
\]
It is clear that $\psi(t_n d (hk)_{n,n})=1$. We obtain that
\[
\psi(\tr(t\theta hk))=\psi(\sum_{j=1}^{n-1}t_j(hk)_{j,j})\cdot\psi(t_1 b(hk)_{j_0,1}+t_{j_0}c(hk)_{1,j_0}).
\]
Hence the integral in \eqref{eq:Theta0} can be written as
\[
\int_{\Theta_\sig}
\psi(\tr(t\theta hk))\ud \theta
=
\psi(\sum_{j=1}^{n-1}t_j(hk)_{j,j})\int_{\Theta_\sig}\psi(t_1 b(hk)_{j_0,1}+t_{j_0}c(hk)_{1,j_0})\ud\theta.
\]
Here $t_1, t_{j_0}\in\Fp^{-\ell}\bs\Fp^{-m}$ with $\ell>m$ sufficiently large, $\alpha=(hk)_{j_0,1}\in\Fo_F^\times$ and $\beta=(hk)_{1,j_0}\in\Fo_F$.
As a topological space, we have
\[
\Theta_\sig\cong\{\begin{pmatrix}1&0\\c&1\end{pmatrix}\}\times\{\begin{pmatrix}1&b\\0&1\end{pmatrix}\}
\]
with $c,b\in\Fp^{\ell_0}$. Hence we obtain that
\[
\int_{\Theta_\sig}\psi(t_1 b(hk)_{j_0,1}+t_{j_0}c(hk)_{1,j_0})\ud\theta
=
\int_{\Theta_\sig'(c)}\bigg(\int_{\Fp^{\ell_0}}\psi(t_1 b(hk)_{j_0,1})\ud b\bigg)\psi(t_{j_0}c(hk)_{1,j_0})\ud\theta'(c).
\]
Since $(hk)_{j_0,1}=(hk)_{\sig(1),1}\in\Fo_F^\times$, and $t_1\in\Fp^{-\ell}\bs\Fp^{-m}$ with $\ell>m>\ell_0$ sufficiently large, this implies that the inner integration
\[
\int_{\Fp^{\ell_0}}\psi(t_1 b(hk)_{j_0,1})\ud b
\]
is zero as long as $\ell>m>\ell_0$ are sufficiently large, due to the choice of $\psi=\psi_F$. The vanishing of the integral in \eqref{eq:Theta0} is proved in this case, and hence for all the cases.
\end{proof}

Recall the kernel distribution in the local theory of Godement-Jacquet is given in \eqref{GJ-kernel}, which is
\[
\Phi_{\GJ}(g)=\psi(\tr g)\cdot|\det g|_F^{\frac{n}{2}}.
\]
As in Lemma \ref{lem:zero-na}, we take the set $\{K_\ell\}_{\ell=1}^\infty$ of the principal congruence subgroups of $K=\RG_n(\Fo_F)$, which forms a family of open compact neighborhoods of the identity $\RI_n$ of $\RG_n(F)$.
We denote by
\begin{align}\label{eq:ck}
\Fc_\ell(g) = \frac{1}{\vol(K_\ell)}
\mathbf{1}_{K_\ell}
\end{align}
the normalized characteristic function of $K_\ell$. It is known that $\{\Fc_\ell\}_{\ell=1}^\infty$ tends to the delta mass supported at $\RI_n$ as $\ell\to\infty$. For $\xi=\Fc_\ell$, the identity in \eqref{convolution} takes the following form.

\begin{lem}\label{lem:pv-1}
For any $g\in\RG_n(F)$, the identity
\[
\left(\Phi_{\GJ}*\Fc_{\ell}^\vee
\right)(g)
=|\det g|_F^{\frac{n}{2}}
\left(\psi(\tr(\cdot))*\Fc_\ell^\vee\right)(g)
\]
holds for all $\ell\geq 1$, where $\Phi_{\GJ}$ is the kernel function as defined in \eqref{GJ-kernel}.
\end{lem}

\begin{proof}
For any $g\in\RG_n(F)$, by definition, we have  that
\begin{align}\label{pv-1}
\left(\Phi_{\GJ}*\Fc_{\ell}^\vee\right)(g)
=
\int_{\RG_n(F)}\psi(\tr h)|\det h|_F^{\frac{n}{2}}\Fc_{\ell}(g^{-1}h)\ud h.
\end{align}
It is clear that $\Fc_{\ell}(g^{-1}h)\neq 0$ if and only if $g^{-1}h\in K_\ell$. Hence
$|\det h|_F=|\det g|_F$. The right-hand side of \eqref{pv-1} can be written as
\begin{align}\label{pv-2}
|\det g|_F^{\frac{n}{2}}
\int_{\RG_n(F)}\psi(\tr h)\Fc_{\ell}(g^{-1}h)\ud h
=
|\det g|_F^{\frac{n}{2}}
\left(\psi(\tr(\cdot))*\Fc_\ell^\vee\right)(g).
\end{align}
This verifies the lemma.
\end{proof}

\begin{lem}\label{lem:k-ell}
For any $\pi\in\Pi_F(n)$ and any $\ell\geq 1$, the integral
\begin{align}\label{k-ell}
\int_{\det g=x}
\left(\Phi_{\GJ}*\Fc_{\ell}^\vee
\right)
(g)\vphi_{\wt{\pi}}(g)\ud_xg
\end{align}
converges absolutely for any $x\in F^\times$, where $\Phi_{\GJ}$ is the kernel function as defined in \eqref{GJ-kernel} and $\vphi_{\wt{\pi}}(g)$ is a matrix coefficient in $\CC(\wt{\pi})$.
\end{lem}

\begin{proof}
By Lemma \ref{lem:pv-1}, we have
\begin{align}\label{pv-3}
\int_{\det g=x}
\left(\Phi_{\GJ}*\Fc_{\ell}^\vee
\right)
(g)\vphi_{\wt{\pi}}(g)\ud_xg
=
|x|_F^{\frac{n}{2}}\int_{\det g=x}
\left(\psi(\tr(\cdot))*\Fc_\ell^\vee\right)(g)
\vphi_{\wt{\pi}}(g)\ud_xg.
\end{align}
Recall from \eqref{eq:ck} that the function $\Fc_\ell$ for $\ell\geq 1$ is the normalized characteristic function of $K_\ell= \RI_n+\RM_n(\Fp^\ell)$ and
the sequence $\{\Fc_\ell\}_{\ell=1}^\infty$ tends to the delta mass supported at $\RI_n$ as $\ell\to\infty$. Consider the convolution of $\psi(\tr(\cdot))$
with $\Fc_\ell$ over $\RG_n(F)$
\[
\psi(\tr(\cdot))*\Fc_\ell^\vee(g):=\int_{\RG_n(F)}\psi(\tr(h))\Fc_\ell(g^{-1}h)\ud h.
\]
If $\Fc_\ell(g^{-1}h)\neq 0$, then we must have that $g^{-1}h\in K_\ell$. By writing $h=g(\RI_n+X)$ with $X\in\RM_n(\Fp^\ell)$, we obtain that $\psi(\tr(h))=\psi(\tr(g))\psi(\tr(gX))$, and $g^{-1}h=\RI_n+X$. It is clear that
$\ud h=\ud^+X$. Hence we obtain that
\[
\psi(\tr(\cdot))*\Fc_\ell^\vee(g)=\psi(\tr(g))\vol(\RM_n(\Fp^\ell))^{-1}\int_{\RM_n(\Fp^\ell)}\psi(\tr(gX))\ud^+X=\psi(\tr(g)){\bf 1}_{\RM_n(\Fp^{-\ell})}(g),
\]
and the identity
\begin{align}\label{eq:psi-c}
\psi(\tr(\cdot))*\Fc_\ell(g)^\vee =
\psi(\tr(g))
\cdot
{\bf 1}_{\RM_n(\Fp^{-\ell})}(g)
\end{align}
holds as functions in $g\in\RG_n(F)$. This implies that the convolution $\psi(\tr(\cdot))*\Fc_\ell^\vee$ with any $\ell\geq 1$ belongs to the space $\CC^\infty_c(\RM_n(F))$. Since the
$\SL_n(F)$-torsor $\RG_n(F)_x$ is closed in $\RM_n(F)$,
we deduce that the integral on the right-hand side of \eqref{pv-3}
is absolutely convergent for each $\ell\geq 1$ and for any $x\in F^\times$. This finishes the proof.
\end{proof}

\begin{lem}\label{lem:stable-na}
For any $\pi\in\Pi_F(n)$, and any given open compact subset $\Fn$ of $F^\times$, the integral in \eqref{k-ell} is uniformly stable as $\ell\to\infty$, when $x$ runs in $\Fn$.
\end{lem}

\begin{proof}
When $n=1$, we have that $G_1(F)=F^\times$. The integral in \eqref{k-ell} is equal to
\[
\left(\Phi_{\GJ}*\Fc_{\ell}^\vee\right)(x)\vphi_{\wt{\pi}}(x).
\]
The assertion in the lemma is obvious. In the following, we assume that $n\geq 2$ and need Lemma \ref{lem:zero-na} for the proof.

By Lemma \ref{lem:pv-1} and the identity in \eqref{pv-3},
we have to show that for any $x\in F^\times$, the following integral
\[
\int_{\det g=x}
\bigg(\psi(\tr(\cdot))*\Fc_{\ell}^\vee
\bigg)
(g)\vphi_{\wt{\pi}}(g)\ud_xg
\]
is independent of $\ell\geq 1$ as long as $\ell$ is sufficiently large, and this stability is uniform when $x$ runs in any given open compact neighborhood $\Fn$ in $F^\times$. In other words, given any
open compact subset $\Fn$ of $F^\times$, there exists a sufficiently large integer $\ell_\Fn$ such that as long as $\ell, m>\ell_\Fn$, we have
\begin{align}\label{eq:stable-2}
\int_{\det g=x}
\bigg(\psi(\tr(\cdot))*\Fc_{\ell}^\vee
\bigg)
(g)\vphi_{\wt{\pi}}(g)\ud_xg
=
\int_{\det g=x}
\bigg(\psi(\tr(\cdot))*\Fc_{m}^\vee
\bigg)
(g)\vphi_{\wt{\pi}}(g)\ud_xg
\end{align}
for all $x\in\Fn$.

In order to prove the equality in \eqref{eq:stable-2}, we consider the convolution
$\psi(\tr(\cdot))*(\Fc_\ell^\vee-\Fc_m^\vee)(g)$.
From \eqref{eq:psi-c}, it is clear that for $\ell>m$, we have
$$
\psi(\tr(\cdot))*(\Fc_\ell^\vee-\Fc_m^\vee)(g) = \psi(\tr(g))
\mathbf{1}_{\RM_n(\Fp^{-\ell})\bs\RM_n(\Fp^{-m})}(g)
=
\psi(\tr(g))\mathbf{1}_{\RM_n(\ell,m)}(g),
$$
where $\RM_n(\ell,m)=\RM_n(\Fp^{-\ell})\bs\RM_n(\Fp^{-m})$.
In order to prove \eqref{eq:stable-2}, it is enough to show that for $\ell>m\geq 1$ sufficiently large, we have
\begin{align}\label{eq:=0}
\int_{\det g=x}
\bigg(
\psi(\tr(\cdot))*(\Fc_\ell^\vee-\Fc_m^\vee)\bigg)(g)
\vphi_{\wt{\pi}}(g)\ud_x g=0.
\end{align}
We write the integral on the left-hand side of $\eqref{eq:=0}$ as
\begin{align}\label{eq:=01}
\int_{\det g=x}
\psi(\tr(g))
\mathbf{1}_{\RM_n(\ell,m)}(g)
\vphi_{\wt{\pi}}(g)\ud_xg
=
\int_{\substack{\det g=x,\\
g\in \RM_n(\ell,m)
}}
\psi(\tr(g))
\vphi_{\wt{\pi}}(g)\ud_xg.
\end{align}
In order to show that the integral on the right-hand side of \eqref{eq:=01} is zero for $\ell>m\geq 1$ sufficiently large, we have to explore the properties of the integrand and the integration domain.

Since $\vphi_{\wt{\pi}}\in \CC(\wt{\pi})$, there exists a positive integer $n_0:=n_{\vphi_{\wt{\pi}}}$, depending on the function $\vphi_{\wt{\pi}}$,
such that $\vphi_{\wt{\pi}}$ is invariant under the left and right translation of the open compact subgroup $K_{n_0}$. Define
$$
K^1_{n_0} =
K_{n_0}\cap \SL_n(F),
$$
which is open and compact in $\SL_n(F)$. Then for any $g\in \RM_n(\ell,m)$ with $\det g=x$, and
$h\in K^1_{n_0}$,
we have that $\det gh=\det hg=\det g=x$ and
\[
 hg, gh\in \RM_n(\ell,m).
\]
Also for the given $\vphi_{\wt{\pi}}(g)$, we have
$$
\vphi_{\wt{\pi}}(g) = \vphi_{\wt{\pi}}(gh) = \vphi_{\wt{\pi}}(hg).
$$
These properties imply that
\begin{align}\label{eq:=02}
\int_{\substack{\det g=x,\\
g\in \RM_n(\ell,m)
}}
\psi(\tr(g))
\vphi_{\wt{\pi}}(g)\ud_xg
=
\int_{\substack{\det g=x,\\
g\in \RM_n(\ell,m)
}}
\psi(\tr(gh))
\vphi_{\wt{\pi}}(g)\ud_xg
\end{align}
for the given $\vphi_{\wt{\pi}}(g)$, because the measure $\ud_xg$ is $\SL_n(F)$-invariant.
By taking the (compact) integration over $h\in K^1_{n_0}$  on both sides of \eqref{eq:=02}, we obtain the following identity:
\begin{align}\label{eq:=03}
\vol(K^1_{n_0})
\int_{\substack{\det g=x,\\
g\in \RM_n(\ell,m)
}}
\psi(\tr(g))
\vphi_{\wt{\pi}}(g)\ud_xg
=
\int_{h\in K^1_{n_0}}\ud h
\int_{\substack{\det g=x,\\
g\in \RM_n(\ell,m)
}}
\psi(\tr(gh))
\vphi_{\wt{\pi}}(g)\ud_xg
\end{align}
for the given $\vphi_{\wt{\pi}}(g)$. Since the double integral on the right-hand side of \eqref{eq:=03} is absolutely convergent, we are able to change the order of the integrations
and write the right-hand side of
\eqref{eq:=03} as
\[
\int_{\substack{\det g=x,\\
g\in \RM_n(\ell,m)
}}
\bigg(\int_{h\in K^1_{n_0}}
\psi(\tr(gh))
\ud h\bigg)
\vphi_{\wt{\pi}}(g)\ud_xg.
\]
By Lemma \ref{lem:zero-na} with $n_0$ sufficiently large, we obtain that the inner integral
\begin{align}\label{eq:stable-3}
\int_{h\in K^1_{n_0}}
\psi(\tr(gh))
\ud h=0
\end{align}
for any $g\in \RM_n(\ell,m)$ with $\det g=x$ and $\ell>m>n_0$ sufficiently large. This proves the equality \eqref{eq:stable-2}. By Lemma \ref{lem:zero-na} again, the vanishing of the integral in \eqref{eq:stable-3} is
uniform for $g\in\RM_n(\ell,m)$ with $\det g=x$ belonging to any given open compact subset $\Fn$ of $F^\times$. This implies that the equality in \eqref{eq:stable-2} holds uniformly when $x$ runs in
any given open compact subset $\Fn$ of $F^\times$.
\end{proof}

\subsection{$\pi$-Kernel function on $\GL_1$}\label{ssec-KF-na}

For any $\pi\in \Pi_F(n)$, we fix $\vphi_{\wt{\pi}} \in \CC(\wt{\pi})$ with $\vphi_{\wt{\pi}}(\RI_n) = 1$. We define the $\pi$-kernel function:
\begin{align}\label{kernel-na}
k_{\pi,\psi}(x) :=
\int^\reg_{\det g=x}
\Phi_{\GJ}(g)\vphi_{\wt{\pi}}(g)\ud_x g=
|x|^{\frac{n}{2}}_F\int^\reg_{\det g=x}\psi(\tr(g))
\vphi_{\wt{\pi}}(g)\ud_x g
\end{align}
where $\Phi_{\GJ}$ is the kernel function as defined in \eqref{GJ-kernel}, the measure
$\ud_x g$, as given in the integral in \eqref{fibration}, is the measure on the fiber
$\RG_n(F)_x$
of the map $\det\colon\GL_n(F)\to F^\times$ as in \eqref{fiber}. Here the regularized integral is explicitly defined by
\begin{align}\label{kernel-pv}
\int^\reg_{\det g=x}\Phi_{\GJ}(g)
\vphi_{\wt{\pi}}(g)\ud_x g
:=
\lim_{\ell\to \infty}
\int_{\det g=x}
\left(\Phi_{\GJ}*\Fc_{\ell}^\vee
\right)
(g)\vphi_{\wt{\pi}}(g)\ud_xg
\end{align}
where the convolution of $\Phi_{\GJ}$ with $\Fc_\ell^\vee$ is taken over $\RG_n(F)$.

\begin{prp}[$\pi$-Kernel Function]\label{k-smooth}
For any $\pi\in \Pi_F(n)$, the $\pi$-kernel function $k_{\pi,\psi}(x)$ defined in \eqref{kernel-na} is a smooth function on $F^\times$.
\end{prp}

\begin{proof}
For any $\pi\in \Pi_F(n)$ and any $\ell>0$, by Lemma \ref{lem:k-ell}, the integral
\[
\int_{\det g=x}
\left(\Phi_{\GJ}*\Fc_{\ell}^\vee
\right)
(g)\vphi_{\wt{\pi}}(g)\ud_xg
\]
converges absolutely for any $x\in F^\times$.
By Lemma \ref{lem:stable-na}, the limit on the right-hand side of \eqref{kernel-pv} is
stably convergent for all $x\in F^\times$ and the convergnce is uniformly stable when $x$ runs in any open compact neighborhood $\Fn$ of $x$ in $F^\times$. Hence the $\pi$-kernel function $k_{\pi,\psi}(x)$ as defined in \eqref{kernel-na} is smooth at any $x\in F^\times$.
\end{proof}

\begin{thm}[Kernel and $\gam$-Function: non-Archimedean case]\label{thm:KF-na}
For any $\pi\in \Pi_F(n)$, the following principal value integral
\begin{align}\label{hankel-na}
\int^\pv_{F^\times}
k_{\pi,\psi}(x)
\chi_s(x^{-1})\ud^\times x
=
\lim_{\ell\to \infty}\sum_{m=-\ell}^\ell
\int_{\RS_m}
k_{\pi,\psi}(x)
\chi_s(x^{-1})\ud^\times x,
\end{align}
with $\RS_m:=\{x\in F^\times\mid |x|_F=q^{-m}\}$, is convergent for $\Re(s)$ sufficiently small, admits a
meromorphic continuation to $s\in \BC$, and is equal to
$
\gam(s+\frac{1}{2},\pi\times \chi,\psi).
$
\end{thm}

\begin{proof}
By Lemma \ref{lem:stable-na},
there exists an integer $\ell_m$, depending on $m$ only, such that for any $\ell>\ell_m$,
$$
k_{\pi,\psi}(x) =
\int_{\det g=x}
\left(\Phi_{\GJ}*\Fc_\ell^\vee\right)(g)
\vphi_{\wt{\pi}}(g)\ud_x g
$$
for any $x\in\RS_m$. Consider the Mellin integral of the kernel function $k_{\pi,\psi}(x)$ along the slice
$\RS_m$:
\begin{align}\label{k-na-1}
\int_{\RS_m}
k_{\pi,\psi}(x)
\chi_s^{-1}(x)\ud^\times x
&=
\int_{\RS_m}\int_{\det g=x}\left(\Phi_{\GJ}*\Fc_\ell^\vee\right)(g)
\vphi_{\wt{\pi}}(g)\ud_x g\ \chi_s^{-1}(x)\ud^\times x\nonumber\\
&=
\int_{\RG_n(F)_m}
\left(\Phi_{\GJ}*\Fc_\ell^\vee\right)(g)
\vphi_{\wt{\pi}[\chi_s^{-1}]}(g)\ud g\nonumber\\
&=
\int_{\RG_n(F)_m}
\int_{\RG_n(F)}\Phi_{\GJ}(h)\Fc_\ell(g^{-1}h)\ud h\
\vphi_{\wt{\pi}[\chi_s^{-1}]}(g)\ud g
\end{align}
where $\vphi_{\wt{\pi}[\chi^{-1}_s]}(g) = \vphi_{\wt{\pi}}(g)\chi_s^{-1}(\det g)$ and for any $m\in\BZ$,
\begin{align}\label{k-na-1-1}
\RG_n(F)_m := \{g\in \RG_n(F)\mid \ |\det g|_F = q^{-m}_F\}.
\end{align}
By definition, we have that $\Fc_\ell(g^{-1}h)\neq 0$ if and only if $g^{-1}h\in K_\ell$. Hence under the
condition that $\Fc_\ell(g^{-1}h)\neq 0$, we have $g\in\RG_n(F)_m$ if and only if $h\in\RG_n(F)_m$.
This implies that the last integral in \eqref{k-na-1} is equal to
\[
\int_{\RG_n(F)_m}
\int_{\RG_n(F)_m}\Phi_{\GJ}(h)\Fc_\ell(g^{-1}h)\ud h\
\vphi_{\wt{\pi}[\chi_s^{-1}]}(g)\ud g
\]
which can also be written as
\begin{align}\label{k-na-2}
\int_{\RG_n(F)}
\int_{\RG_n(F)_m}\Phi_{\GJ}(h)\Fc_\ell(g^{-1}h)\ud h\
\vphi_{\wt{\pi}[\chi_s^{-1}]}(g)\ud g
\end{align}
For $g\in\RG_n(F)_m$, we have
\begin{align*}
\int_{\RG_n(F)_m}\Phi_{\GJ}(h)\Fc_\ell(g^{-1}h)\ud h
=
\int_{\RG_n(F)}{\bf 1}_{\RG_n(F)_m}(g)\Phi_{\GJ}(h)\Fc_\ell(g^{-1}h)\ud h
=
\left(\Phi_{\GJ,m}*\Fc^\vee_\ell\right)(g),
\end{align*}
where ${\bf 1}_{\RG_n(F)_m}(g)$ is the characteristic function of $\RG_n(F)_m$, and
$\Phi_{\GJ,m}(g)$ is defined to be
\begin{align}\label{k-na-2-1}
\Phi_{\GJ,m}(g):={\bf 1}_{\RG_n(F)_m}(g)\Phi_{\GJ}(h).
\end{align}
By \cite[Proposition 2.8]{JL21}, the function $\Phi_{\GJ,m}(g)$ defines an invariant, essentially
compact distribution on $\RG_n(F)$ and hence belongs to the Bernstein center $\FZ(\RG_n(F))$
of $\RG_n(F)$.
Thus, together with \eqref{k-na-2}, we obtain that
\begin{align}\label{k-na-3}
\int_{\RS_m}
k_{\pi,\psi}(x)
\chi_s^{-1}(x)\ud^\times x
=
\int_{\RG_n(F)}
\left(\Phi_{\GJ,m}*\Fc^\vee_\ell\right)(g)
\vphi_{\wt{\pi}[\chi_s^{-1}]}(g)\ud g,
\end{align}
which is equal to $\CZ(1-(s+\frac{1}{2}),\CF_{\GJ,m}(\Fc_\ell),\vphi_{\pi[\chi]}^\vee)$,
where $\CF_{\GJ,m}(\Fc_\ell)(g)=(\Phi_{\GJ,m}*\Fc^\vee_\ell)(g)$ as in
\cite[Equation 2.18]{JL21}.
By the functional equation in \cite[Equation (2.24)]{JL21}, we have
\[
\CZ(1-s,\CF_{\GJ,m}(\xi),\vphi^\vee_{\pi[\chi]})
=
f_{\GJ,m}(\wt{\pi_{\chi_{s-\frac{1}{2}}}})\cdot
\CZ(s,\xi,\vphi_{\pi[\chi]})
\]
for any $\xi\in\CC_c^\infty(\RG_n(F))$, $\vphi_\pi\in\CC(\pi)$ and $\chi\in\FX(F^\times)$,
where the function $f_{\GJ,m}$ denotes the regular function on the Bernstein variety
$\Ome(\RG_n(F))$ of $\RG_n(F)$.
Together with \eqref{k-na-3}, we obtain that
\begin{align}\label{k-na-4}
\int_{\RS_m}
k_{\pi,\psi}(x)
\chi_s^{-1}(x)\ud^\times x
=
f_{\GJ,m}(\wt{\pi_{\chi_s}})\cdot
\CZ(s+\frac{1}{2},\Fc_\ell,\vphi_{\pi[\chi]}).
\end{align}
By definition, we have
\[
\CZ(s+\frac{1}{2},\Fc_\ell,\vphi_{\pi[\chi]})
=
\int_{\RG_n(F)}\Fc_\ell(g)\vphi_\pi(g)\chi(\det g)|\det g|^s\ud g.
\]
When $\ell>\ell_m$ is sufficiently large, $K_\ell=\RI_n+\RM_n(\Fp^\ell)$ is a sufficiently small
open compact (principal congruence) subgroup of $\RG_n(\Fo_F)$. Hence we have
$\vphi_\pi(g)=\vphi_\pi(\RI_n)=1$, and $\det g\in 1+\Fp^{\ell}$. This implies that
$|\det g|_F=1$ and $\chi(\det g)=1$. We obtain that
\[
\CZ(s+\frac{1}{2},\Fc_\ell,\vphi_{\pi[\chi]})
=
\int_{\RG_n(F)}\Fc_\ell(g)\ud g=1.
\]
Therefore, we obtain that
\begin{align}\label{k-na-5}
\int_{\RS_m}
k_{\pi,\psi}(x)
\chi_s^{-1}(x)\ud^\times x
=
f_{\GJ,m}(\wt{\pi_{\chi_s}})
\end{align}
for all $m\in\BZ$, and
\[
\int^\pv_{F^\times}
k_{\pi,\psi}(x)
\chi_s(x^{-1})\ud^\times x
=
\lim_{\ell\to \infty}\sum_{m=-\ell}^\ell
\int_{\RS_m}
k_{\pi,\psi}(x)
\chi_s(x^{-1})\ud^\times x
=
\lim_{\ell\to \infty}\sum_{m=-\ell}^\ell
f_{\GJ,m}(\wt{\pi_{\chi_s}}).
\]
By \cite[Proposition 2.8, Part (2)]{JL21},  we know that the following limit exists
\[
\lim_{\ell\to \infty}\sum_{m=-\ell}^\ell
f_{\GJ,m}(\wt{\pi_{\chi_s}})
\]
for $\Re(s)$ sufficiently small,  admits a meromorphic continuation to all $s\in\BC$,
and is equal to
\[
\gam(\frac{1}{2},\pi_{\chi_s}.\psi)=\gam(s+\frac{1}{2},\pi\times\chi,\psi).
\]
We are done.
\end{proof}

\begin{cor}\label{kernel-wd}
For any $\pi\in\Pi_F(n)$, as a distribution on $F^\times$, the $\pi$-kernel function
$k_{\pi,\psi}(x)$ is independent of the choice of the matrix coefficient
$\vphi_{\wt{\pi}} \in \CC(\wt{\pi})$ with $\vphi_{\wt{\pi}}(\RI_n) = 1$ and the chosen sequence
$\{\Fc_\ell\}_{\ell=1}^\infty$ that tends to the delta mass supported at $\RI_n$.
\end{cor}

\begin{proof}
On the one hand, for the contragredient $\wt{\pi}$ of $\pi\in\Pi_F(n)$ and for any $m\in\BZ$, we can view the function
\[
f_{\GJ,m}(\wt{\pi_{\chi_s}})
=
f_{\GJ,m}(\wt{\pi}\otimes(\chi_s)^{-1})
\]
as function of the quasi-character $\chi_s^{-1}$. On the other hand, we can write
\[
\int_{\RS_m}
k_{\pi,\psi}(x)
\chi_s(x)^{-1}\ud^\times x
=
\int_{F^\times}k_{\pi,\psi}(x){\bf 1}_{\RS_m}(x)
\chi_s(x)^{-1}\ud^\times x.
\]
The right-hand side can be viewed as the action of the function $(k_{\pi,\psi}\cdot{\bf 1}_{\RS_m})(x)$ on the quasi-character (irreducible admissible representation)
$\chi_s^{-1}$ of $\RG_1(F)=F^\times$. Hence by the identity in \eqref{k-na-5}, the function $(k_{\pi,\psi}\cdot{\bf 1}_{\RS_m})(x)$ belongs to the Berstein center $\FZ(\RG_1(F))$ of $\RG_1(F)$.
By \cite{Ber84}, the function $(k_{\pi,\psi}\cdot{\bf 1}_{\RS_m})(x)$ does not depend on the choice of the matrix coefficient $\vphi_{\wt{\pi}}(g)$ with $\vphi_{\wt{\pi}}(\RI_n)=1$,
but depends only on $\wt{\pi}$.

We can deduce that the kernel function $k_{\pi,\psi}(x)$ is independent of the chosen sequence
$\{\Fc_\ell\}_{\ell=1}^\infty$ that tends to the delta mass supported at the identity $\RI_n$
by using the Mellin
inversion formula for $f\in\CC_c^\infty(F^\times)$. If we define the kernel function $k_{\pi,\psi}(x)$ as a limit
by using any filtration of open compact neighborhoods at the identity, since the filtration of the principal
congruence subgroups $\{K_\ell\}_{\ell=1}^\infty$ forms a basis of open compact neighborhoods of
the identity of $\RG_n(F)$, we are able to choose a sub-fitration $\{K_{\ell_j}\}_{j=1}^\infty$ of
$\{K_\ell\}_{\ell=1}^\infty$, which gives the same limit for the kernel function $k_{\pi,\psi}(x)$.
We give a detailed proof for the Archimedean case (the proof of Part (3) of Theorem \ref{thm:KF-ar}).
Hence the kernel function $k_{\pi,\psi}(x)$ depends only on  $\wt{\pi}$ and the additive character $\psi=\psi_F$.
\end{proof}

\section{$\pi$-Kernel Function: Archimedean case}
\label{sec-KF-ar}

\subsection{$\pi$-Kernel function on $\GL_1$}\label{ssec-KF-ar}

Let $F$ be an Archimedean local field, which is either $\BR$ or $\BC$. In this case, we recall that $\Pi_F(n)$ is the set of equivalence classes of irreducible Casselman-Wallach representations of $\RG_n(F)$. Recall from \eqref{K} that $K$ is the maximal compact subgroup of $\RG_n(F)$, which is $\RO(n)$ if $F=\BR$ and $\RU(n)$ if $F=\BC$.

In order to define the $\pi$-kernel function $k_{\pi,\psi}(x)$ on $F^\times$ as in the non-Archimedean case (Section \ref{ssec-KF-na}), we
choose a sequence of test functions $\{\Fc_\ell\}_{\ell=1}^\infty\subset \CC^\infty_c(\RG_n(F))$, such that for any smooth function $h\in \CC_c^\infty(\RG_n(F))$,
\begin{align}\label{Fc-I}
\lim_{\ell\to \infty}
\int_{\RG_n(F)}
\Fc_\ell(g)h(g) \ud g= h(\RI_n).
\end{align}
In other words, the sequence $\{\Fc_\ell\}_{\ell=1}^\infty$ tends to the delta mass supported at the identity $\RI_n$. After multiplying the sequence by a partition function, it will be convenient to assume that the support of $\Fc_\ell$ is contained in a given compact neighborhood $\FN$ of $\RI_n$ for all $\ell$, and for any smooth $h\in \CC^\infty(\RG_n(F))$, the identity in \eqref{Fc-I} holds. In fact, it is not hard to verify that for any function $h(g)$ on $\RG_n(F)$, which is continuous in a neighborhood of the identity $\RI_n$, the identity in \eqref{Fc-I} still holds.

From \eqref{FOFT}, for any $\xi\in\CS_{\std}(\RG_n(F))$ with
$\xi(g)=|\det g|_F^{\frac{n}{2}}\cdot f(g)$ for some $f\in\CS(\RM_n(F))$,
the following identity
\[
\CF_{\GJ}(\xi)(g)=\left(\Phi_{\GJ}*\xi^\vee\right)(g)
=|\det g|_F^{\frac{n}{2}}\cdot\CF_\psi(f)(g)=
|\det g|_F^{\frac{n}{2}}\cdot\CF_\psi(|\det (\cdot)|^{-\frac{n}{2}}\xi)(g)
\]
holds. Take any $\vphi_{\wt{\pi}}\in \CC(\wt{\pi})$ with $\vphi_{\wt{\pi}}(\RI_n) = 1$, and define,
as in the non-Archimedean case, the $\pi$-kernel function
\begin{equation}\label{kernel-ar}
k_{\pi,\psi}(x)
:=
\int^\reg_{\det g=x}
\Phi_{\GJ}(g)\vphi_{\wt{\pi}}(g)\ud_x g=
|x|^{\frac{n}{2}}_F\int^\reg_{\det g=x}\psi(\tr(g))
\vphi_{\wt{\pi}}(g)\ud_x g
\end{equation}
where $\Phi_{\GJ}$ is the kernel function as defined in \eqref{GJ-kernel}.
The integral is regularized as follows:
\begin{align}\label{k-reg-1}
\int^\reg_{\det g=x}
\Phi_{\GJ}(g)\vphi_{\wt{\pi}}(g)\ud_x g
:=
\lim_{\ell\to \infty}
\int_{\det g=x}
\left(\Phi_{\GJ}*\Fc_\ell^\vee\right)(g)
\vphi_{\wt{\pi}}(g)
\ud_xg.
\end{align}
Applying \eqref{convolution} to the case that $\xi=\Fc_\ell$, we have
\[
\left(\Phi_{\GJ}*\Fc_\ell^\vee\right)(g)
=
|\det g|_F^{\frac{n}{2}}\left(\psi(\tr(\cdot))*(|\det(\cdot)|_F^{\frac{n}{2}}\Fc_\ell)^\vee\right)(g).
\]
The regularized integral in \eqref{k-reg-1} can also be written as
\begin{align}\label{k-reg-2}
\int^\reg_{\det g=x}
\Phi_{\GJ}(g)\vphi_{\wt{\pi}}(g)\ud_x g
=
|x|^{\frac{n}{2}}_F
\lim_{\ell\to \infty}
\int_{\det g=x}
\left(\psi(\tr(\cdot))*(|\det(\cdot)|_F^{\frac{n}{2}}\Fc_\ell)^\vee\right)(g)
\vphi_{\wt{\pi}}(g)
\ud_xg.
\end{align}
For each $\ell\geq 1$, by \eqref{FOFT}, we have
\begin{align}\label{k-reg-3}
\left(\Phi_{\GJ}*\Fc_\ell^\vee\right)(g)
=
|\det g|_F^{\frac{n}{2}}\cdot\CF_\psi(|\det (\cdot)|^{-\frac{n}{2}}\Fc_\ell)(g).
\end{align}
Since the support of $\Fc_\ell\in\CC_c^\infty(\RG_n(F))$ is contained in a given compact neighborhood $\FN$ of identity $\RI_n$ for every $\ell$, the function
$|\det g|^{-\frac{n}{2}}\Fc_\ell(g)$ belongs to $\CC_c^\infty(\RG_n(F))$. Hence
the classical Fourier transform $\CF_\psi(|\det (\cdot)|^{-\frac{n}{2}}\Fc_\ell)(g)$ is a Schwartz function on $\RM_n(F)$ for each $\ell$. By definition, $\left(\Phi_{\GJ}*\Fc_\ell^\vee\right)(g)\in \CS_\mathrm{std}(\RG_n(F))$.
For any $x\in F^\times$, its restriction to the closed submanifold $\RG_n(F)_x$ of
$\left(\Phi_{\GJ}*\Fc_\ell^\vee\right)(g)$ is still a Schwartz function in the sense of \cite{AG08}.
As $\pi$ is an irreducible
Casselman-Wallach representation of $\RG_n(F)$, the matrix coefficient $\vphi_{\wt{\pi}}$ is of moderate growth on $\RG_n(F)$, so is its restriction to the closed submanifold $\RG_n(F)_x$.
We deduce  that for each $\ell\geq 1$,
the following integral
\begin{equation}\label{ar-k-1}
\int_{\det g=x}
\left(\Phi_{\GJ}*\Fc_\ell^\vee\right)(g)
\vphi_{\wt{\pi}}(g)
\ud_xg
\end{equation}
is absolutely convergent. After choosing a section to the determinant morphism:
$$
t_1: F^\times\to \GL_n(F)
$$
sending $x\in F^\times$ to $t_1(x):=\diag(x,1,...,1)$, the integral \eqref{ar-k-1} can be rewritten as
$$
\int_{\SL_n(F)}
\left(
\Phi_{\GJ}*\Fc_\ell^\vee
\right)(t_1(x)g)
\vphi_{\wt{\pi}}(t_1(x)g)
\ud_1g
$$
which converges uniformly when $x$ runs in a given compact subset $\Fn$ in $F^\times$. Hence \eqref{ar-k-1} defines a smooth function in $x\in F^\times$. We summarize the discussion as the following proposition.

\begin{prp}\label{prp:k-ell-ar}
For any $\pi\in\Pi_F(n)$ and any $\ell\geq 1$, the $\ell$-th part of the kernel function
\[
k_{\pi,\psi,\ell}(x)
:=
\int_{\det g=x}
\left(\Phi_{\GJ}*\Fc_\ell^\vee\right)(g)
\vphi_{\wt{\pi}}(g)
\ud_xg,
\]
which is defined on the right-hand side of \eqref{k-reg-1},
converges absolutely for any $x\in F^\times$. Moreover, for any compact subset $\Fn$ of $F^\times$,
the convergence is uniform when $x$ runs in $\Fn$ and the integral defines a smooth function in $x\in F^\times$.
\end{prp}

For each $\ell\geq 1$, from \eqref{convolution} with $\xi=\Fc_\ell$, the $\ell$-th part of the kernel function can be written as
\begin{align}\label{ell-k}
k_{\pi,\psi,\ell}(x)
=
|x|^{\frac{n}{2}}_F
\int_{\det g=x}
\left(
\psi(\tr(\cdot))*(|\det(\cdot)|_F^{\frac{n}{2}}\Fc_\ell)^\vee
\right)(g)
\vphi_{\wt{\pi}}(g)
\ud_xg.
\end{align}
By Proposition \ref{prp:k-ell-ar}, the $\ell$-th part of the kernel function $k_{\pi,\psi,\ell}(x)$ is
smooth on $F^\times$. Moreover, we have

\begin{prp}\label{prp:MT-k-ell}
For any given $\pi\in\Pi_F(n)$, there exists a real number $s_\pi$ such that the Mellin transform of the $\ell$-th part of the kernel function $k_{\pi,\psi,\ell}(x)$,
$$
\int_{F^\times}
k_{\pi,\psi,\ell}(x)\chi_s(x)^{-1}\ud^\times x,
$$
converges absolutely for $\Re(s)<s_\pi$, and admits a meromorphic continuation to $s\in \BC$, for any $\ell\geq 1$. Here $\chi_s(x)=\chi(x)|x|^s$ for any unitary character $\chi\in\FX(F^\times)$.
\end{prp}

\begin{proof}
For each integer $\ell\geq 1$ and any unitary character $\chi$ of $F^\times$, we consider the Mellin transform of the $\ell$-th part of the kernal function, $k_{\pi,\psi,\ell}(x)$:
\begin{align}\label{2-1}
\int_{F^\times}
k_{\pi,\psi,\ell}(x)
\chi_s(x)^{-1}\ud^\times x
=
\int_{F^\times}
\int_{\det g=x}
\left(\Phi_{\GJ}*\Fc_\ell^\vee\right)(g)
\vphi_{\wt{\pi}}(g)
\ud_xg\
\chi_s(x)^{-1}\ud^\times x,
\end{align}
which can be written as
\begin{align}\label{2-2}
\int_{\RG_n(F)}
\CF_{\GJ}(\Fc_\ell)(g)
\vphi_{\wt{\pi}}(g)
\chi(\det g)^{-1}|\det g|^{-s}\ud g
=
\CZ(1-(s+\frac{1}{2}),\CF_{\GJ}(\Fc_\ell),\vphi_\pi^\vee,\chi^{-1}).  		
\end{align}
Since $\CC_c^\infty(\RG_n(F))$ is a subspace of $\CS_{\std}(\RG_n(F))$, we must have that
$\CF_{\GJ}(\Fc_\ell)\in\CS_{\std}(\RG_n(F))$. This implies that the integral in \eqref{2-2}
converges absolutely for $\Re(s)$ sufficiently small, and up to unramified shift, the $\ell$-th part of the kernel function, $k_{\pi,\psi,\ell}(x)$ belongs to the space $\CS_{\wt{\pi}}(F^\times)$.
Hence up to unramified shift, $k_{\pi,\psi,\ell}(x)$ belongs to the space $\CM^{-1}(\CL_{\wt{\pi}})$,
according to \cite[Corollary 3.8]{JL21}.

We are going to show that the integral in \eqref{2-2} converges absolutely for $\Re(s)<s_\pi$ with
some constant $s_\pi\in \BR$, which is independent of $\ell$ and $\chi$, but depends only on $\pi$.
By \cite[Theorem 2.3]{JL21} and the structure of the space $\CL_{\wt{\pi}}(\FX(F^\times))\subset\CZ(\FX(F^\times))$ in \cite[Definition 2.2]{JL21},
it suffices to show that after meromorphic continuation, as a meromorphic function in $s\in \BC$, the poles of the integral in \eqref{2-2} are contained in the region
$\{s\in \BC\mid \Re(s)>s_\pi\}$ for the constant $s_\pi\in \BR$. Note that the Mellin transform in \eqref{2-1} uses $\chi_s^{-1}$.

By the functional equation of Godement-Jacquet in \eqref{GJFE}, which comes from
\cite[Proposition 2.7]{JL21}, after meromorphic continuation, the integral in \eqref{2-2} is equal to
\begin{align}\label{2-3}
\gam(s+\frac{1}{2},\pi\times \chi,\psi)
\cdot
\CZ(s+\frac{1}{2},\Fc_\ell,\vphi_\pi,\chi).
\end{align}
Since $\Fc_\ell(g)\in\CC_c^\infty(\RG_n(F))$, the integral in \eqref{2-3} converges absolutely for all
$s\in \BC$ and hence defines a holomorphic function in $s\in \BC$.  Thus the set of the poles of the product in \eqref{2-3} is contained in the set of the poles of the gamma function
$\gam(s+\frac{1}{2},\pi\times \chi,\psi)$.
As meromorphic functions in $s$, we have
$$
\gam(s+\frac{1}{2},\pi\times \chi,\psi) =
\veps(s+\frac{1}{2},\pi\times \chi,\psi)
\cdot
\frac{L(\frac{1}{2}-s,\wt{\pi}\times \chi^{-1},\psi)}{
L(s+\frac{1}{2},\pi\times \chi,\psi)
}.
$$
The set of the poles of the integral in \eqref{2-2} is contained in the set of the poles of
the local $L$-function
$$
L(\frac{1}{2}-s,\wt{\pi}\times \chi^{-1},\psi).
$$
Finally, since $\chi$ is assumed to be unitary, by the explicit formulas of the Archimedean local $L$-functions (see \cite[(3.6),(4.6)]{Kn94} for instance),
there exists a constant $s_\pi$ depending on $\pi$ only, such that the poles of $L(1-s,\wt{\pi}\times \chi^{-1},\psi)$ are contained in the half-plane
$$
\{s\in \BC\mid \Re(s)>s_\pi\}.
$$
It follows that the integral in \eqref{2-2} is absolutely convergent for $\Re(s)<s_\pi$ with $s_\pi$ depending only on $\pi$, neither on $\chi$ nor on $\ell\geq 1$.

Finally, it is clear that from \eqref{2-2}, the assertion on the meromorphic continuation follows from that of the zeta integral on the right-hand side of \eqref{2-2} by \cite[Theorem 3.4]{JL21}.
\end{proof}

\begin{prp}\label{prp:k-ar}
For any $\pi\in\Pi_F(n)$, the limit in \eqref{k-reg-1} defining the kernel function $k_{\pi,\psi}(x)$ converges absolutely for any $x\in F^\times$. Moreover, for any compact subset $\Fn$ of $F^\times$, the convergence of the limit is uniform when $x$ runs in $\Fn$.
\end{prp}

\begin{proof}
From \eqref{kernel-ar}, \eqref{k-reg-1} and \eqref{k-reg-2}, we have
\begin{align}\label{k-ar-1}
k_{\pi,\psi}(x)
&=
\int^\reg_{\det g=x}
\Phi_{\GJ}(g)\vphi_{\wt{\pi}}(g)\ud_x g
=
\lim_{\ell\to \infty}
\int_{\det g=x}
\left(\Phi_{\GJ}*\Fc_\ell^\vee\right)(g)
\vphi_{\wt{\pi}}(g)
\ud_xg\nonumber\\
&=|x|^{\frac{n}{2}}_F
\lim_{\ell\to \infty}
\int_{\det g=x}
\left(\psi(\tr(\cdot))*(|\det(\cdot)|_F^{\frac{n}{2}}\Fc_\ell)^\vee\right)(g)
\vphi_{\wt{\pi}}(g)
\ud_xg.
\end{align}

We are going to apply the elliptic regularity the Casimir operator $\Del$ on $\RG_n(F)$ as in \cite[Lemma~3.7]{BeK14} to the matrix coefficient $\vphi_{\wt{\pi}}(g)$ in \eqref{k-ar-1}
as the first step in the proof of the proposition. The Lie algebra of $\RG_n$ is $\RM_n$. We may simply identify $\RM_n$ with its dual and take
\[
\Del=\sum_{i,j=1}^nX_{i,j}^2
\]
as in \cite[Section 3.2.1]{BeK14}, where $\{X_{i,j}\mid i,j=1,2\cdots,n\}$ is a basis of $\RM_n$. It is well known that $\Del$ is well defined independent of the choice of a basis of $\RM_n$
and belongs to the center of the universal enveloping algebra of $\RG_n$.
By the elliptic regularity as in \cite[Lemma~3.7]{BeK14}, for any integer $m$ with $2m>\dim \RG_n$,
there exist functions $f_1$ and $f_2$ with $f_1\in \CC^{2m-\dim \RG_n-1}_c(\RG_n(F))$ and $f_2\in \CC^\infty_c(\RG_n(F))$, respectively, such that
\begin{align}\label{Del}
\Del^m*f_1+f_2 = \del_{\RI_n}.
\end{align}
Let $L$ be the left translation action of $\RG_n(F)$ on a space of functions on $\RG_n(F)$. From \eqref{Del}, we have
\[
L(\Del^m*f_1+f_2)(f)(g)=L(\del_{\RI_n})(f)(g)=f(g)
\]
for any suitable function $f$ on $\RG_n(F)$.

Following \eqref{Del} and the definition appearing in \cite[\S 3.3]{BeK14}, or \cite[Corollary~3.8]{BeK14}, we have that
\[
\vphi_{\wt{\pi}}=L(\Del^m*f_1+f_2)(\vphi_{\wt{\pi}})
=L(f_1)(\Del^m\cdot \vphi_{\wt{\pi}})+L(f_2)(\vphi_{\wt{\pi}})
=c(m,\wt{\pi})\cdot L(f_1)(\vphi_{\wt{\pi}})
+L(f_2)(\vphi_{\wt{\pi}})
\]
where $c(m,\wt{\pi})\in \BC$ is a constant depending on $m$ and $\wt{\pi}$. Here we use the fact that $\Del$ acts on $\vphi_{\wt{\pi}}$ via scalar.
After substituting
\[
\vphi_{\wt{\pi}}\mapsto L(\Del^m*f_1+f_2)(\vphi_{\wt{\pi}}),
\]
in the integral in \eqref{k-ar-1},
up to scalar, it suffices to investigate the behavior of the following integrals when $\ell\to \infty$,
\begin{align}\label{kernel-ar-3}
\int_{\det g=x}
\bigg(
\psi(\tr(\cdot))
*
(|\det(\cdot)|^{\frac{n}{2}}_F\Fc_\ell)^\vee
\bigg)(g)
L(f_i)(\vphi_{\wt{\pi}})(g)\ud_xg, \quad i=1,2
\end{align}
where $f_1\in \CC^{2m-\dim \RG_n-1}_c(\RG_n(F))$ and $f_2\in \CC^\infty_c(\RG_n(F))$.

Following the same argument as in the paragraph after \eqref{k-reg-3}, for each individual $\ell\geq 1$, the integral \eqref{kernel-ar-3} is absolutely convergent. Expanding the definition of $L(f_i)(\vphi_{\wt{\pi}})$, we get
\[
L(f_i)(\vphi_{\wt{\pi}})(g)
=
\int_{h\in \RG_n(F)}
f_i(h)\vphi_{\wt{\pi}}(h^{-1}g)\ud h
=
\int_{h\in \RG_n(F)}
f_i(gh)\vphi_{\wt{\pi}}(h^{-1})\ud h.
\]
We write the integrals in \eqref{kernel-ar-3} as
\begin{align}\label{ar-0}
\int_{\det g=x}
\bigg(
\psi(\tr(\cdot))
*
(|\det(\cdot)|^{\frac{n}{2}}_F\Fc_\ell)^\vee
\bigg)(g)
\int_{h\in \RG_n(F)}
f_i(gh)\vphi_{\wt{\pi}}(h^{-1})\ud h\ud_xg, \quad i=1,2.
\end{align}
By absolute convergence we can switch the integration order in \eqref{ar-0} and write \eqref{ar-0} as
\begin{align}\label{ar-1}
\int_{h\in \RG_n(F)}
\vphi_{\wt{\pi}}(h^{-1})
\ud h
\int_{\det g=x}
\bigg(
\psi(\tr(\cdot))
*
(|\det(\cdot)|^{\frac{n}{2}}_F\Fc_\ell)^\vee
\bigg)(g)
f_i(gh)\ud_xg, \quad i=1,2.
\end{align}

The second step is to show that the limit as $\ell\to\infty$ exists. To this end, we consider the function given by the inner integration in \eqref{ar-1}, which can be written as follows:
\begin{align}\label{ar-2}
h\in \RG_n(F)\mapsto
&\int_{\det g=x}
\bigg(
\psi(\tr(\cdot))
*
(|\det(\cdot)|^{\frac{n}{2}}_F\Fc_\ell)^\vee
\bigg)(g)
f_i(gh)\ud_xg \nonumber
\\
=&
\int_{\det g=x}
\int_{u\in \RG_n(F)}
\psi(\tr(gu))
|\det u|^{\frac{n}{2}}_F
\Fc_\ell(u)\ud u
f_i(gh)\ud_xg.
\end{align}
Since both $\Fc_\ell$ and $f_i$ are compactly supported, one can switch the integration order, and obtain that the function in \eqref{ar-2} is equal to the following function
\begin{align}
h\in \RG_n(F)
\mapsto
\int_{u\in \RG_n(F)}
\Fc_\ell(u)|\det u|^{\frac{n}{2}}_F\ud u
\int_{\det g=x}
\psi(\tr(gu))f_i(gh)\ud_xg.
\end{align}
Hence \eqref{ar-1} becomes
\begin{align}
\int_{h\in \RG_n(F)}
\vphi_{\wt{\pi}}(h^{-1})\ud h
\int_{u\in \RG_n(F)}
\Fc_\ell(u)|\det u|^{\frac{n}{2}}_F\ud u
\int_{\det g=x}
\psi(\tr(gu))f_i(gh)\ud_xg.
\end{align}
By absolute convergence, after switching the integration order, \eqref{ar-1} can be re-written as
\begin{align*}
\int_{u\in \RG_n(F)}
\Fc_\ell(u)|\det u|^{\frac{n}{2}}_F\ud u
\int_{h\in \RG_n(F)}
\vphi_{\wt{\pi}}(h^{-1})\ud h
\int_{\det g=x}
\psi(\tr(gu))f_i(gh)\ud_xg.
\end{align*}
Notice that $\{\Fc_\ell\cdot |\det(\cdot)|^{\frac{n}{2}}_F\}$ is a delta sequence consisting of smooth compactly supported functions tending to the delta mass supported at identity. By the discussion right after \eqref{Fc-I}, in order to show that as $\ell\to \infty$ the limit exists, it suffices to show that the following function
$$
u\in \RG_n(F)
\mapsto
\int_{h\in \RG_n(F)}
\vphi_{\wt{\pi}}(h^{-1})\ud h
\int_{\det g=x}
\psi(\tr(gu))f_i(gh)\ud_xg
$$
is a continuous function in a compact neighborhood of identity. After changing variable, the above function becomes
\begin{align}\label{ar-2-0}
u\in \RG_n(F)
\mapsto
\int_{h\in \RG_n(F)}
\vphi_{\wt{\pi}}(h^{-1}u^{-1})\ud h
\int_{\det g=x\det u}
\psi(\tr(g))f_i(gh)\ud_xg,
\end{align}
which we are going to show to be absolutely convergent whenever $u$ lies in a fixed compact neighborhood of identity. By taking absolute value, we consider
$$
\int_{h\in \RG_n(F)}
|\vphi_{\wt{\pi}}(h^{-1}u^{-1})|\ud h
\cdot
\bigg|
\int_{\det g=x\det u}
\psi(\tr(g))f_i(gh)\ud_xg
\bigg|
$$
as a function in $u$ near the identity $\RI_n$.
Since the matrix coefficients $\vphi_{\wt{\pi}}$ are smooth on $\RG_n(F)$, up to a constant that depends on the compact neighborhood containing $u$, it suffices to estimate the following integral
\begin{equation}\label{eq:domin:archi}
\int_{h\in \RG_n(F)}
|\vphi_{\wt{\pi}}(h^{-1})|\ud h
\cdot
\bigg|
\int_{\det g=x}
\psi(\tr(g))f_i(gh)\ud_xg
\bigg|
\end{equation}
for $x$ lying in a compact neighborhood of $1\in \RG_1(F)$.

We first estimate the function from the inner integration in \eqref{eq:domin:archi}, which is given by
\begin{equation}\label{ar-2-1}
h\in \RG_n(F)\mapsto
\int_{\det g=x}
\psi(\tr(g))f_i(gh)\ud_xg,\quad i=1,2.
\end{equation}
With the chosen section $t_1:\RG_1(F)\to \RG_n(F)$ sending $x\mapsto t_1(x)=\diag(x,1,...,1)$, the function \eqref{ar-2-1} can be re-written as
\begin{align}\label{ar-2-2}
h\in \RG_n(F)\mapsto
\int_{\det g=1}
\psi(\tr(t_1(x)g))f_i(t_1(x)gh)\ud_1g,\quad i=1,2.
\end{align}
In particular, via $t_1$, $\RG_n(F) \simeq \RG_1(F)\rtimes \SL_n(F)$. Hence as a function on a smooth manifold, we may view \eqref{ar-2-2} as a function in variable $(y,h_1)\in \RG_1(F)\rtimes \SL_n(F)$, i.e.
\begin{align}\label{ar-2-3}
(y,h_1)\in \RG_1(F)\rtimes \SL_n(F)
\mapsto
\int_{\det g=1}
\psi(\tr(t_1(x)g))f_i(t_1(x)gt_1(y)h_1)\ud_1g,\quad i=1,2.
\end{align}
Changing variable $g\mapsto t_1(y)gt_1(y)^{-1}$ and using the fact that $\psi(\tr(\cdot))$ is adjoint invariant, the function becomes
\begin{align}\label{ar-2-4}
(y,h_1)\in \RG_1(F)\rtimes \SL_n(F)
\mapsto
&\int_{\det g=1}
\psi(\tr(t_1(x)g))f_i(t_1(xy)gh_1)\ud_1g \nonumber
\\
=&
\int_{\det g=1}
\psi(\tr(t_1(x)g))\wt{f}_i(xy,gh_1)\ud_1g,\quad i=1,2
\end{align}
where $\wt{f}_i$ is the pull-back of $f_i$ under the identification $\RG_n(F) = \RG_1(F)\rtimes \SL_n(F)$. In particular, as long as $x$ is fixed (or lies inside a fixed compact subset of $\RG_1(F)$), the function \eqref{ar-2-4} is compactly supported in $y$-variable uniformly for any $h_1\in \SL_n(F)$.

Now we consider the matrix coefficient $\vphi_{\wt{\pi}}$. Since $\vphi_{\wt{\pi}}$ is of moderate growth, there exists a positive integer $n_{\wt{\pi}}$ depending on $\wt{\pi}$ only, such that up to constant,
\begin{align}\label{npi}
|\vphi_{\wt{\pi}}(h)|
\leq
\begin{cases}
(1+\tr({}^thh)+\tr({}^th^{-1}h^{-1}))^{n_{\wt{\pi}}}, & F=\BR,\\
(1+\tr({}^t\wb{h}h)+\tr({}^t\wb{h}^{-1}h^{-1}))^{n_{\wt{\pi}}}, & F=\BC.
\end{cases}
\end{align}
Write $h = t_1(y)\cdot h_1$, then
$
|\vphi_{\wt{\pi}}(h)|
$
is essentially bounded by a finite linear combination of the following terms
$$
y^{p_{\wt{\pi}}}\cdot P_{\wt{\pi}}(h_1)
$$
where $p_{\wt{\pi}}\in \BZ$ depends on $\wt{\pi}$, and $P_{\wt{\pi}}(h_1)$ is the restriction of a polynomial on $\RM_n(F)$ (viewed as a real vector space) depending on $\wt{\pi}$. It follows that in order to estimate \eqref{eq:domin:archi}, up to a finite linear combination, it suffices to estimate the following integral
$$
\int_{y\in \RG_1(F)}
y^{-p_{\wt{\pi}}}\ud y
\int_{h_1\in \SL_n(F)}
P_{\wt{\pi}}(h_1^{-1})
\ud_1h_1
\int_{\det g=1}
\psi(\tr(t_1(x)g))\wt{f}_i(xy,gh_1)\ud_1g,\quad i=1,2.
$$
As the integrand in $y$-variable is compactly supported whenever $x$ lies in a fixed compact subset of $\RG_1(F)$, it suffices to show the convergence of the following integral
\begin{align}\label{ar-4}
&\int_{h_1\in \SL_n(F)}
P_{\wt{\pi}}(h_1^{-1})
\ud_1h_1
\int_{\det g=1}
\psi(\tr(t_1(x)g))\wt{f}_i(xy,gh_1^{-1})\ud_1g
\nonumber\\
=&
\int_{h_1\in \SL_n(F)}
P_{\wt{\pi}}(h_1^{-1})
\ud_1h_1
\int_{\det g=1}
\psi(\tr(t_1(x)gh_1))\wt{f}_i(xy,g)\ud_1g,\quad i=1,2.
\end{align}
Here $f_1\in \CC^{2m-\dim \RG_n-1}_c(\RG_n(F))$, $f_2\in \CC^\infty_c(\RG_n(F))$ and $m$ can be arbitrary large. For $h_1\in \SL_n(F)$, $h^{-1}_1 = \mathrm{adj}(h_1)$ where $\mathrm{adj}(h_1)$ is the adjacent matrix of $h_1$. This implies that $P_{\wt{\pi}}(h^{-1}_1)$, when restricted to $\SL_n(F)$, can still be represented by a polynomial (depending only on $\wt{\pi}$) in $h_1$-variable. It is clear that the desired convergence of the integral in \eqref{ar-4} will follow from Lemma \ref{lem:auxiliarycov} below, which forms the third step of the proof.

\begin{lem}\label{lem:auxiliarycov}
 Given any positive integer $k$, a function $\wt{f}\in \CC^k_c(\RG_1(F)\rtimes \SL_n(F))$ and $x\in \Fn$ with $\Fn$ any compact subset of $F^\times$ near $1$.
 Then for any polynomial $P$ on $\RM_n(F)$ (viewed as a real vector space) of degree smaller than or equal to $k$, there exists a constant $C_{P,\Fn,f}$ depending only on $P,\Fn$ and $f$, such that the following inequality holds for any $(y,h_1)\in \RG_1(F)\times \SL_n(F)$,
$$
\bigg|
P(h_1)\cdot
\int_{\det g=1}
\psi(\tr(t_1(x)gh_1))\wt{f}(xy,g)\ud_1g
\bigg|
\leq C_{P,\Fn,f}.
$$
\end{lem}

\begin{proof}
By elementary calculation, there exists a linear differential operator $\frac{\partial^{P,x}}{\partial g}$, which has its degree the same as that of $P$ and has constant coefficients where the constants depend on $P$ and linearly on $x$, such that
$$
P(h_1)
\cdot
\psi((\tr(t_1(x)gh_1))) =
\frac{\partial^{P,x}
\psi(\tr(t_1(x)(\cdot )h_1))
}{\partial g}(g).
$$
Therefore the LHS of the above inequality becomes
\begin{align}\label{ar-5}
\int_{\det g=1}
\frac{\partial^{P,x}
\psi(\tr(t_1(x)(\cdot )h_1))
}{\partial g}(g)
\cdot
\wt{f}(xy,g)\ud_1g.
\end{align}
According to \cite[Theorem~18.1.34]{H94} or from basic Euclidean analysis, the adjoint linear differential operator of $\frac{\partial^{P,x}}{\partial g}$ exists.  Note that \cite[Theorem~18.1.34]{H94} ensures the existence of the adjoint as a pseudo differential operator. But locally the adjoint is represented by a linear differential operator by basic Euclidean analysis. Thus, the adjoint is indeed a linear differential operator. Overall, by adjunction, there exists an adjoint differential operator $\frac{{}^t\partial^{P,x}}{\partial g}$ of the same degree as $P$, with smooth coefficients linearly depending on $x\in F^\times$, such that the integral in \eqref{ar-5} is equal to
$$
\int_{\det g=1}
\psi(\tr(t_1(x)gh_1))
\cdot
\frac{{}^t\partial^{P,x}
\wt{f}(xy,\cdot)
}{\partial g}(g)
\ud_1g.
$$
Using the fact that $\wt{f}\in \CC^k_{c}(\RG_1(F)\rtimes \SL_n(F))$, the result follows.
\end{proof}

Finally, by choosing the positive integer $k$ in Lemma \ref{lem:auxiliarycov} sufficiently large comparing to the degree of $P_{\wt{\pi}}(h_1^{-1})$ as a polynomial in $h_1\in\SL_n(F)$,
the integral in \eqref{ar-4} must be absolutely convergent. This finishes the proof of the Proposition.
\end{proof}

\begin{cor}[$\pi$-Kernel Function]\label{cor:kernel}
For any given $\pi\in\Pi_F(n)$,  the kernel function $k_{\pi,\psi}(x)$ defined via the regularized integral  in \eqref{kernel-ar} is a smooth function on $F^\times$.
\end{cor}

\begin{proof}
By Proposition \ref{prp:k-ar}, the limit in \eqref{kernel-ar} and \eqref{k-reg-1}
that defines the kernel function $k_{\pi,\psi}(x)$ is uniform when $x$ runs in any compact
neighborhood $\Fn$ in $F^\times$. Moreover, from \eqref{ar-2-2} and the fact that $f_1\in \CC^{2m-\dim \RG_n-1}_c(\RG_n(F))$, $f_2\in \CC^\infty_c(\RG_n(F))$ where $m$ can be of arbitrarily large, the derivative in $x$-variable can be of arbitrary high order, hence the kernel function $k_{\pi,\psi}(x)$ is smooth
on $F^\times$.
\end{proof}

Here is the main result in this section.

\begin{thm}[Kernel and $\gam$-Function: Archimedean case]\label{thm:KF-ar}
For any given $\pi\in\Pi_F(n)$,  let $s_\pi$ be the real number as in Proposition \ref{prp:MT-k-ell}. The following principal value integral
\begin{equation}\label{eq:archi:regularizedlimit}
\int^\pv_{F^\times}k_{\pi,\psi}(x)\chi_s(x^{-1})\ud^\times x
:=\lim_{\ell\to \infty}
\int_{F^\times}
k_{\pi,\psi,\ell}(x)\chi_s(x^{-1})\ud^\times x
\end{equation}
converges for $\Re(s)<s_\pi$, admits a meromorphic continuation to $s\in\BC$, and is equal to the local gamma function $\gam(s+\frac{1}{2},\pi\times \chi,\psi)$ associated to $(\pi,\chi)$. Moreover,
as a distribution on $F^\times$, $k_{\pi,\psi}(x)$ is independent of the choice of the matrix coefficient $\vphi_{\wt{\pi}}$ with $\vphi_{\wt{\pi}}(\RI_n) = 1$ and the chosen sequence
$\{\Fc_\ell\}_{\ell=1}^\infty$ that tends to the delta mass supported at $\RI_n$.
\end{thm}

\begin{proof}
We take the real number $s_\pi$ as in Proposition \ref{prp:MT-k-ell}. By \eqref{2-2} and \eqref{2-3}, whenever $\Re(s)<s_\pi$, as a meromorphic function in $s\in \BC$, we have
\begin{align}\label{2-5}
\int_{F^\times}
k_{\pi,\psi,\ell}(x)
\chi_s(x)^{-1}\ud^\times x
=\gam(s+\frac{1}{2},\pi\times \chi,\psi_F)
\cdot\CZ(s+\frac{1}{2},\Fc_\ell,\vphi_\pi,\chi).
\end{align}
Again, as $\Fc_\ell\in \CC^\infty_c(\RG_n(F))$, the integral on the right-hand side of \eqref{2-5}
converges absolutely for all $s\in \BC$. By the assumption that  the support of
$\Fc_\ell$ is contained in a given compact neighborhood $\FN$ of $\RI_n$ for all $\ell$, after taking the limit as $\ell\to \infty$, we obtain the identity
$$
\lim_{\ell\to \infty}
\int_{F^\times}
k_{\pi,\psi,\ell}(x)\chi_s(x)^{-1}
\ud^\times x = \gam(s+\frac{1}{2},\pi\times \chi,\psi)
$$
as $\vphi_\pi(\RI_n)=\vphi_{\wt{\pi}}(\RI_n)=1$. It is clear that this identity holds as meromorphic functions over $s\in\BC$.

As a distribution on $F^\times$, for $f\in \CC^\infty_c(F^\times)$, we consider the convolution of the kernel function $k_{\pi,\psi}(x)$ with $f^\vee$:
\begin{align}\label{3-1}
(k_{\pi,\psi}*f^\vee)(x)
=
\int_{F^\times}
k_{\pi,\psi}(y)f(x^{-1}y)\ud^\times y
=
\lim_{\ell\to \infty}
\int_{F^\times}
k_{\pi,\psi,\ell}(y)f(x^{-1}y)\ud^\times y.
\end{align}
As a distribution, the pairing of $k_{\pi,\psi}(x)$ and $f(x)$ is given by
\[
(k_{\pi,\psi},f)=(k_{\pi,\psi}*f^\vee)(1).
\]
Write $f^x(y):=f(x^{-1}y)$ in the rest of the proof. By \cite[Theorem 2.3]{JL21}, for any $s_0<s_\pi$ (where $s_\pi$ is as in Proposition \ref{prp:MT-k-ell}), we have the following Mellin inversion formula:
\begin{align}\label{3-2}
f^x(y) =
\sum_{\ome\in\Ome^\wedge}
\frac{1}{2\pi i}
\int_{\Re(s) = s_0}
\CM(f^x)(|\cdot|^s\ome(\cdot))
|y|^{-s}
\ome(y)^{-1}
\ud s
\end{align}
where $\Ome^\wedge$ is the Pontryagin dual of $\Ome$, where if $F=\BR$, $\Ome=\{\pm 1\}$, and if $F=\BC$, $\Ome = \BC^\times_1$ which is the set of norm $1$ elements in $\BC$. In particular $\ome(y)=\ac(y)^p$ with $p\in\BZ/2\BZ$ if $F=\BR$, and $p\in\BZ$ if $F=\BC$. Note that the right-hand side of \eqref{3-2} is absolutely convergent for any $s_0<s_\pi$.
The right-hand side of \eqref{3-1} can be written as
\begin{align}\label{3-3}
\lim_{\ell\to \infty}
\int_{F^\times}
k_{\pi,\psi,\ell}(y)
\sum_{\ome}
\frac{1}{2\pi i}
\int_{\Re(s) = s_0}
\CM(f^x)(|\cdot|^s \ome(\cdot))
|y|^{-s}\ome(y)^{-1}
\ud s
\ud^\times y.
\end{align}
Note that the integral is absolutely convergent for any $s_0<s_\pi$. By exchanging the order of integrations, we obtain that the expression in \eqref{3-3} is equal to
\begin{align}\label{3-4}
\lim_{\ell\to \infty}
\sum_{\ome}
\frac{1}{2\pi i}
\int_{\Re(s) = s_0}
\CM(f^x)(|\cdot|^s \ome(\cdot))
\int_{F^\times}
k_{\pi,\psi,\ell}(y)
|y|^{-s}\ome(y)^{-1}
\ud^\times y
\ud s.
\end{align}
When $\Re(s) = s_0<s_\pi$, we apply the formula in \eqref{2-5} to the inner integration of \eqref{3-4} and obtain
\begin{align}\label{3-5}
\int_{F^\times}
k_{\pi,\psi,\ell}(y)
|y|^{-s}\ome(y)^{-1}
\ud^\times y
=
\gam(s+\frac{1}{2},\pi\times \ome,\psi)
\CZ(s+\frac{1}{2},\Fc_\ell,\vphi_{\pi[\ome]})
\end{align}
Since the sequence $\{\Fc_\ell\}$ tends to the delta mass supported at $\RI_n$ as $\ell\to\infty$, and the support of $\Fc_\ell$ is contained in a given compact neighborhood $\FN$ of $\RI_n$
for all $\ell\geq 1$, we deduce that
$$
\CZ(s+\frac{1}{2},\Fc_\ell,\vphi_{\pi[\ome]})
=
\int_{\RG_n(F)}
\Fc_\ell(g)
\vphi_{\pi[\ome]}(g)
|\det g|^{s-\frac{1}{2}}_F\ud g
$$
is uniformly bounded (independent of $\ell$ and $\Re(s) = s_0$). By exchanging the order of the limit with integrations, we obtain that
\[
\lim_{\ell\to \infty}\int_{\RG_n(F)}
\Fc_\ell(g)
\vphi_{\pi[\ome]}(g)
|\det g|^{s-\frac{1}{2}}_F\ud g=1
\]
as $\vphi_\pi(\RI_n)=\vphi_{\wt{\pi}}(\RI_n)=1$, and
the expression in \eqref{3-4} is equal to
\begin{align}\label{3-6}
\sum_{\ome}
\frac{1}{2\pi i}
\int_{\Re(s) = s_0}
\CM(f^x)(|\cdot|^s \ome(\cdot))
\gam(s+\frac{1}{2},\pi\times \ome,\psi)
\ud s.
\end{align}
It is clear that the expression in \eqref{3-6} is independent of the choice of matrix coefficients
$\vphi_{\wt{\pi}}\in\CC(\wt{\pi})$ with $\vphi_{\wt{\pi}}(\RI_n)=1$, and the choice of such sequences $\{\Fc_\ell\}_{\ell=1}^\infty$.
Therefore, the integral in \eqref{3-1} enjoys the same property, which implies that as a distribution on $F^\times$, the kernel function $k_{\pi,\psi}(x)$ depends only on $\pi$ and
$\psi=\psi_F$. This finishes the theorem.
\end{proof}

\section{Fourier Operator and Hankel Transform}\label{sec-FOHT}

In this section we get back to take $F$ to be any local field of characteristic zero. Recall from \eqref{kernel-na}, \eqref{kernel-pv}, \eqref{kernel-ar}, and \eqref{k-reg-1},
that for any $\pi\in\Pi_F(n)$
and for a given additive character $\psi=\psi_F$,
the kernel function $k_{\pi,\psi}(x)$ is given by the following formula:
\begin{equation}\label{kernel-local}
k_{\pi,\psi}(x)
=
\int^\reg_{\det g=x}\Phi_{\GJ}(g)
\vphi_{\wt{\pi}}(g)\ud_x g
=
\lim_{\ell\to \infty}
\int_{\det g=x}
\left(\Phi_{\GJ}*\Fc_{\ell}^\vee
\right)
(g)\vphi_{\wt{\pi}}(g)\ud_xg
\end{equation}
where $\vphi_{\wt{\pi}}\in\CC(\wt{\pi})$ with $\vphi_{\wt{\pi}}(\RI_n)=1$, $\{\Fc_\ell\}_{\ell\geq 1}$ is a sequence of test functions in $\CC^\infty_c(\RG_n(F))$ that
tends to the delta mass supported at $\RI_n$ as $\ell\to\infty$ and the support of $\Fc_\ell$ is contained in a given compact neighborhood $\FN$ of $\RI_n$ for all $\ell\geq 1$ (see \eqref{eq:ck} for
the non-Archimedean case and \eqref{Fc-I} for the Archimedean case). By \cite[Corollary 3.8]{JL21},
we have
\[
\CC_c^\infty(F^\times)\subset \CS_\pi(F^\times)\subset\CC^\infty(F^\times)
\]
for any $\pi\in\Pi_F(n)$. It is proved in Corollary \ref{kernel-wd} and Theorem \ref{thm:KF-ar} that as a distribution on $\CC_c^\infty(F^\times)$,
the kernel function $k_{\pi,\psi}(x)$ is independent of the choice of $\vphi_{\wt{\pi}}\in\CC(\wt{\pi})$  and $\{\Fc_\ell\}_{\ell\geq 1}\subset\CC^\infty_c(\RG_n(F))$.

\subsection{Hankel transform}\label{ssec-HT}
In this section, we are going to show that the Fourier operator $\CF_{\pi,\psi}$ as defined in \eqref{eq:1-FO} can be represented as a Hankel transform (convolution operator) with the kernel function $k_{\pi,\psi}(x)$ over the space of test functions $\CC_c^\infty(F^\times)$.

\begin{thm}[Hankel Transform]\label{thm:HT}
Let $F$ be any local field of characteristic zero.
For any $\pi\in\Pi_F(n)$, the Fourier operator $\CF_{\pi,\psi}$ defined in \eqref{eq:1-FO} can be represented as a convolution operator by the kernel function $k_{\pi,\psi}(x)$
\[
\CF_{\pi,\psi}(\phi_0)(x)=(k_{\pi,\psi}*\phi_0^\vee)(x)
\]
for any $\phi_0\in\CC_c^\infty(F^\times)$.
\end{thm}

\begin{proof}
As in \eqref{3-1}, we write
\begin{align}\label{5-1}
(k_{\pi,\psi}*\phi_0^\vee)(x)
=
\int_{F^\times}
k_{\pi,\psi}(y)\phi_0(x^{-1}y)\ud^\times y
=
\int_{F^\times}
k_{\pi,\psi}(y)\phi_0^x(y)\ud^\times y,
\end{align}
with $\phi_0^x(y)=\phi_0(x^{-1}y)$ for any $\phi_0\in\CC_c^\infty(F^\times)$. It is clear that the integral converges absolutely since the kernel function $k_{\pi,\psi}(y)$ is a smooth function on $F^\times$
(Proposition \ref{k-smooth} for the non-Archimedean case and Corollary \ref{cor:kernel} for the Archimedean case).

We first consider the case that $F$ is Archimedean. By using the Archimedean Mellin inversion formula for $\phi_0^x(y)$ as in \eqref{3-2}, we write the right-hand side of \eqref{5-1} as
\begin{align}\label{5-2}
\int_{F^\times}
k_{\pi,\psi}(y)
\sum_{\ome\in\Ome^\wedge}
\frac{1}{2\pi i}
\int_{\Re(s) = s_0}
\CM(\phi_0^x)(|\cdot|_F^s\ome(\cdot))
|y|_F^{-s}
\ome(y)^{-1}
\ud s\ud^\times y,
\end{align}
which is absolutely convergent for any $s_0<s_\pi$ as in \eqref{3-3}. By the same arguments used from \eqref{3-2} to \eqref{3-6}, we obtain that
\begin{align}\label{5-3}
(k_{\pi,\psi}*\phi_0^\vee)(x)
=
\sum_{\ome\in\Ome^\wedge}
\frac{1}{2\pi i}
\int_{\Re(s) = s_0}
\CM(\phi_0^x)(|\cdot|_F^s\ome(\cdot))
\gam(s+\frac{1}{2},\pi\times \ome,\psi)\ud s
\end{align}
with $s_0<s_\pi$. From the definition of the Mellin transform, we obtain that
\begin{align*}
\CM(\phi_0^x)(|\cdot|_F^s\ome(\cdot))
=
|x|^{s}_F
\ome(x)
\CM(\phi_0)(|\cdot|^{s}_F\ome(\cdot)).
\end{align*}
Hence we obtain that the right-hand side of \eqref{5-3} is equal to
\begin{align}\label{5-4}
\sum_{\ome\in\Ome^\wedge}
\frac{1}{2\pi i}
\int_{\Re(s) = s_0}
\CM(\phi_0)(|\cdot|_F^s\ome(\cdot))
\gam(s+\frac{1}{2},\pi\times \ome,\psi)
|x|^{s}_F\ome(x)\ud s.
\end{align}
By definition, we have
\[
\CM(\phi_0)(|\cdot|_F^s\ome(\cdot))
=
\CZ(s+\frac{1}{2},\phi_0,\ome).
\]
Then we obtain
\begin{align*}
\CM(\phi_0)(|\cdot|_F^s\ome(\cdot))
\gam(s+\frac{1}{2},\pi\times \ome,\psi)
&=
\gam(s+\frac{1}{2},\pi\times \ome,\psi)
\CZ(s+\frac{1}{2},\phi_0,\ome)\\
&=
\CZ(\frac{1}{2}-s,\CF_{\pi,\psi}(\phi_0),\ome^{-1}).
\end{align*}
The last identity uses the functional equation in \eqref{GL1-FE}, which comes from \cite[Theorem 3.10]{JL21}. Similarly, we have
\[
\CZ(\frac{1}{2}-s,\CF_{\pi,\psi}(\phi_0),\ome^{-1})=
\CM(\CF_{\pi,\psi}(\phi_0))(|\cdot|^{-s}\ome^{-1}).
\]
When $s_0<s_\pi$, \eqref{5-4} is equal to
\[
\sum_{\ome\in\Ome^\wedge}
\frac{1}{2\pi i}
\int_{\Re(s) = s_0}
\CM(\CF_{\pi,\psi}(\phi_0))(|\cdot|^{-s}\ome^{-1})
|x|^{s}_F\ome(x)\ud s=\CF_{\pi,\psi}(\phi_0)(x),
\]
according to Mellin inversion formula in \cite[Theorem 2.3]{JL21}. Therefore, we obtain
\[
\CF_{\pi,\psi}(\phi_0)(x)
=
(k_{\pi,\psi}*\phi_0^\vee)(x)
\]
as functions on $F^\times$. This proves the theorem for the Archimedean case.

For the non-Archimedean case, the only difference is the expression of the Mellin inversion formula for
$\phi_0^x$. In the Archimedean case, it was given in \eqref{3-2}, while the non-Archimedean case is given
by
\begin{align}
\phi_0^x(y)=
\sum_{\ome\in\Ome^\wedge}\bigg(\Res_{s=0}(\CM(\phi_0^x)(|\cdot|_F^s\ome(\cdot))|x|_F^{-s}q^s)\bigg)\ome(\ac(x))^{-1}.
\end{align}
Note that for a given function $\phi_0^x(y)$, the summation over $\ome\in\Ome^\wedge$ only takes finitely many characters $\ome$. It is clear now that the proof for the Archimedean case
works in the same way for the non-Archimedean case. We omit the details here.
\end{proof}

\subsection{Homogeneous distributions on $\GL_1$ and local $\gam$-functions}\label{ssec-HDGF}

Let $\chi_s(x)=\chi(x)\cdot|x|_F^s$ be a quasi-character of $F^\times$ with $\chi$ a unitary
character of $F^\times$ and $s\in\BC$. One may regard $\chi_s$ as a distribution (generalized function)
on $F^\times$ by defining
\begin{align}\label{chis-dist}
(\chi_s,\phi_0):=\int_{F^\times}\phi_0(x)\chi_s(x)\ud^\times x
\end{align}
for any $\phi_0\in \CC_c^\infty(F^\times)$. As usual, one defines the Fourier transform of the
distribution $\chi_s$ by
\begin{align}\label{FT-chis}
(\CF_{\pi,\psi}(\chi_s),\phi_0):=(\chi_s,\CF_{\pi,\psi}(\phi_0))
\end{align}
for any $\phi_0\in \CC_c^\infty(F^\times)$. We may compute the distribution
$\CF_{\pi,\psi}(\chi_s^{-1})$ as follows.
\begin{align*}
(\CF_{\pi,\psi}(\chi_s^{-1}),\phi_0)
&=(\chi_s^{-1},\CF_{\pi,\psi}(\phi_0))\\
&=\int_{F^\times}\CF_{\pi,\psi}(\phi_0)(x)\chi_s^{-1}(x)\ud^\times x\\
&=\CZ(\frac{1}{2}-s,\CF_{\pi,\psi}(\phi_0),\chi^{-1}).
\end{align*}
As in the proof of Theorem \ref{thm:HT}, by using the functional equation in \eqref{GL1-FE},
which comes from \cite[Theorem 3.10]{JL21}, we obtain that
\begin{align*}
\CZ(\frac{1}{2}-s,\CF_{\pi,\psi}(\phi_0),\chi^{-1})
&=
\CZ(s+\frac{1}{2},\phi_0,\chi)\cdot \gam(\frac{1}{2},\pi\times\chi_s,\psi)\\
&= (\chi_s,\phi_0)\cdot\gam(\frac{1}{2},\pi\times\chi_s,\psi).
\end{align*}
Hence we obtain that
\[
(\CF_{\pi,\psi}(\chi_s^{-1}),\phi_0)=(\chi_s,\phi_0)\cdot\gam(\frac{1}{2},\pi\times\chi_s,\psi)
\]
for all $\phi_0\in \CC_c^\infty(F^\times)$, which holds for all $s\in\BC$ by meromorphic continuation. Therefore, we deduce that
as distributions on $F^\times$, the following functional equation
\begin{align}\label{FE-chis}
\CF_{\pi,\psi}(\chi_s^{-1})=\gam(\frac{1}{2},\pi\times\chi_s,\psi)\cdot\chi_s
\end{align}
holds for $s\in\BC$ after meromorphic continuation.

\begin{thm}\label{GGPS-gam}
For any $\pi\in\Pi_F(n)$, the local Langlands $\gam$-functions $\gam(s,\pi\times\chi,\psi)$ with any
unitary characters $\chi$ of $F^\times$ are the gamma functions in the sense of Gelfand, Graev and
Piatetski-Shapiro in \cite{GGPS} and of Weil in \cite{W66}, i.e.
\[
\CF_{\pi,\psi}(\chi_s^{-1})=\gam(\frac{1}{2},\pi\times\chi_s,\psi)\cdot\chi_s
\]
holds for $s\in\BC$ after meromorphic continuation. Moreover, the Fourier transform
$\CF_{\pi,\psi}(\chi_s)$ of the homogeneous distribution $\chi_s$ on $F^\times$ can be expressed
as a convolution integral with $k_{\pi,\psi}(x)$ as the kernel function:
\[
\CF_{\pi,\psi}(\chi_s)(x)=(k_{\pi,\psi}*\chi_s^\vee)(x),
\]
which holds for $s\in\BC$ after meromorphic continuation.
\end{thm}

\begin{proof}
By Theorems \ref{thm:KF-na} and \ref{thm:KF-ar}, we have
\begin{align*}
(k_{\pi,\psi}*\chi_s)(x)
&=
\int^{\pv}_{F^\times}k_{\pi,\psi}(y)\chi_s(y^{-1}x)\ud^\times y=
\chi_s(x)\int^{\pv}_{F^\times}k_{\pi,\psi}(y)\chi_s(y^{-1})\ud^\times y\\
&=\chi_s(x)\cdot(k_{\pi,\psi}*\chi_s)(1)=\chi_s(x)\cdot\gam(\frac{1}{2},\pi\times\chi_s,\psi)
=\CF_{\pi,\psi}(\chi_s^{-1})(x).
\end{align*}
Hence we obtain the desired identity:
\[
\CF_{\pi,\psi}(\chi_s)(x)=(k_{\pi,\psi}*\chi_s^\vee)(x),
\]
which holds for $s\in\BC$ after meromorphic continuation. Note that $\chi_s^{-1}(x)=\chi_s^\vee(x)$.
\end{proof}


\begin{thebibliography}{jiang2020}

\bibitem[AG08]{AG08}
Aizenbud, A.; Gourevitch, D.
{\it Schwartz functions on Nash manifolds}. Int. Math. Res. Not. IMRN 2008, no. 5, Art. ID rnm 155, 37 pp.


\bibitem[Ar13]{Ar13}
Arthur, James
{\it The endoscopic classification of representations. Orthogonal and symplectic groups}.
American Mathematical Society Colloquium Publications, 61. American Mathematical Society, Providence, RI, 2013.

\bibitem[Ber84]{Ber84}
Bernstein, J.
{\it Le ``centre" de Bernstein.}
(French) Edited by P. Deligne. Representations of reductive groups over a local field, Travaux en Cours, Hermann, Paris, 1984, 1-32.


\bibitem[BeK14]{BeK14}
Bernstein, Joseph; Kr\"otz, Bernhard
{\it Smooth Fréchet globalizations of Harish-Chandra modules}.
Israel J. Math. 199 (2014), no. 1, 45--111.



\bibitem[BK00]{BK00}
Braverman, A.; Kazhdan, D.
{\it $\gamma$-functions of representations and lifting.}
With an appendix by V. Vologodsky. GAFA 2000 (Tel Aviv, 1999).
Geom. Funct. Anal. 2000, Special Volume, Part I, 237--278.



\bibitem[C79]{C79}Cartier, P. {\it Representations of $p$-adic groups: A survey}. Automorphic forms, representations and $L$-functions, Part 1, pp. 111--155, Proc. Sympos. Pure Math., XXXIII, Amer. Math. Soc., Providence, R.I., 1979.


\bibitem[C89]{C89}
Casselman, W.
{\it Canonical extensions of Harish-Chandra modules}. Can. J. of Math. 41(1989), 315--438.



\bibitem[FS21]{FS21}
Fargues, L; Scholze, P.
{\it Geometrization of the local Langlands correspondence}.
arXiv:2102.13459.

\bibitem[GGPS]{GGPS}
Gelfand, I. M.; Graev, M. I.; Pyatetskii-Shapiro, I. I.
{\it Generalized functions.} Vol. 6.
Representation theory and automorphic functions. Translated from the 1966 Russian original by K. A. Hirsch. Reprint of the 1969 English translation.
AMS Chelsea Publishing, Providence, RI, 2016.


\bibitem[GJ72]{GJ72}
Godement, R.; Jacquet, H.
{\it Zeta functions of simple algebras.}
Lecture Notes in Mathematics, Vol. 260. Springer-Verlag, Berlin-New York, 1972. ix+188 pp.




\bibitem[H94]{H94}
H\"{o}rmander, L.
{\it The analysis of linear partial differential operators. {III}}. Classics in Mathematics, viii+525, Springer, Berlin, 1994.


\bibitem[H00]{H00}
Henniart, G.
{\it Une preuve simple des conjectures de Langlands pour $\GL(n)$ sur un corps $p$-adique}. (French) Invent. Math. 139 (2000), no. 2, 439--455.

\bibitem[HT01]{HT01}
Harris, M.; Taylor, R.
{\it The geometry and cohomology of some simple Shimura varieties. With an appendix by Vladimir G. Berkovich}. Annals of Mathematics Studies, 151. Princeton University Press, Princeton, NJ, 2001.




\bibitem[JL21]{JL21}
Jiang, Dihua; Luo, Zhilin
{\it Certain Fourier operators and their associated Poisson summation formulae on $\GL_1$}.
arXiv:2108.03566.




\bibitem[Kn94]{Kn94}
Knapp, A. W.
{\it Local Langlands correspondence: the Archimedean case}. Motives (Seattle, WA, 1991), 393--410, Proc. Sympos. Pure Math., 55, Part 2, Amer. Math. Soc., Providence, RI, 1994.



\bibitem[L70]{L70}
Langlands, R.
{\it Problems in the theory of automorphic forms}. Lectures in modern analysis and applications, III, pp. 18--61. Lecture Notes in Math., Vol. 170, Springer, Berlin, 1970.



\bibitem[L89]{L89}
Langlands, R.
{\it On the classification of irreducible representations of real algebraic groups}. Representation theory and harmonic analysis on semisimple Lie groups, 101--170, Math. Surveys Monogr., 31, Amer. Math. Soc., Providence, RI, 1989.











\bibitem[N20]{N20}
Ng\^o, B. C.
{\it Hankel transform, Langlands functoriality and functional equation of automorphic $L$-functions}. Jpn. J. Math. 15 (2020), no. 1, 121--167.

\bibitem[S63]{S63}
Satake, I.
{\it Theory of spherical functions on reductive algebraic groups over $p$-adic fields}.
Inst. Hautes \'Etudes Sci. Publ. Math. 1963, no. 18, 5--69.






\bibitem[Sc13]{Sc13}
Scholze, P.
{\it The local Langlands correspondence for $\GL_n$ over $p$-adic fields}.
Invent. Math. 192 (2013), no. 3, 663--715.

\bibitem[SZ11]{SZ11}
Sun, Binyong; Zhu, Chen-Bo
{\it A general form of Gelfand-Kazhdan criterion}. Manuscripta Math. 136 (2011), no. 1-2, 185--197.



\bibitem[T50]{Tt50}
Tate, J.
{\it Fourier analysis in number fields and Hecke's zeta-functions.} Thesis (Ph.D.), Princeton University. 1950. 60 pp; Algebraic Number Theory (Proc. Instructional Conf., Brighton, 1965), 305--347, Thompson, Washington, D.C., 1967.

\bibitem[Wal88]{Wal88}
Wallach, N. Real reductive groups. I. Pure and Applied Mathematics, 132-I. Academic Press, Inc., Boston, MA, 1988.

\bibitem[Wal92]{Wal92}
Wallach, N. Real reductive groups. II. Pure and Applied Mathematics, 132-II. Academic Press, Inc., Boston, MA, 1992.

\bibitem[W66]{W66}
Weil, A. {\it Fonction z\^eta et distributions}. (French)
S\'eminaire N. Bourbaki, 1966, exp. no 312, 523--531.





\end{thebibliography}
\end{document}